\documentclass{amsart}

\usepackage{amsmath,amssymb,amsthm}
\usepackage{mathrsfs}
\usepackage{verbatim}
\usepackage{graphicx}
\usepackage[percent]{overpic}
\usepackage{fullpage}
\usepackage[colorlinks,linkcolor=blue,urlcolor=blue,citecolor=blue,destlabel=true]{hyperref}
\usepackage{mathtools}
\usepackage{tikz}
\usetikzlibrary{matrix,arrows}
\usepackage{color}
\usepackage{comment}
\usepackage{enumerate}
\usepackage{caption}

\newtheoremstyle{myremark} 
{7pt}                    
{7pt}                    
{}  	                 
{}                           
{\bf}       	         
{.}                          
{.5em}                       
{}  

\theoremstyle{plain}
\newtheorem{lemma}{Lemma}
\newtheorem{theorem}[lemma]{Theorem}
\newtheorem*{theorem-main}{Main Theorem}
\newtheorem*{theorem-miss-origin}{Theorem~\ref{thm:miss-origin}}
\newtheorem*{theorem-bu}{Theorem~\ref{thm:bu}}
\newtheorem*{theorem-bu-Sn}{Theorem~\ref{thm:bu-Sn}}
\newtheorem*{theorem-bu-Sn-2}{Theorem~\ref{thm:bu-Sn-2}}
\newtheorem{corollary}[lemma]{Corollary}
\newtheorem{proposition}[lemma]{Proposition}
\newtheorem{conjecture}[lemma]{Conjecture}

\theoremstyle{definition}
\newtheorem{definition}[lemma]{Definition}

\theoremstyle{myremark}
\newtheorem{remark}[lemma]{Remark}

\newcommand{\N}{\mathbb{N}}

\newcommand{\R}{\mathbb{R}}

\newcommand{\Z}{\mathbb{Z}}
\newcommand{\cB}{\mathcal{B}}

\newcommand{\sP}{\mathscr{P}}

\newcommand{\diam}{\mathrm{diam}}

\newcommand{\conv}{\mathrm{conv}}
\newcommand{\cone}{\mathrm{cone}}

\newcommand{\sm}{\mathrm{SM}}

\newcommand{\vr}[2]{\mathrm{VR}(#1;#2)}
\newcommand{\vrm}[2]{\mathrm{VR}^m(#1;#2)}

\newcommand{\sign}{\mathrm{sign}}

\newcommand{\supp}{\mathrm{supp}}
\newcommand{\ut}{\underline{t}}

\begin{document}

\title[Metric thickenings, Borsuk--Ulam theorems, and orbitopes]{Metric thickenings, Borsuk--Ulam theorems, and orbitopes}
\author{Henry Adams}
\address[HA]{Department of Mathematics, Colorado State University, Fort Collins, CO 80523, United States}
\email{adams@math.colostate.edu}
\author{Johnathan Bush}
\address[JB]{Department of Mathematics, Colorado State University, Fort Collins, CO 80523, United States}
\email{bush@math.colostate.edu}
\author{Florian Frick}
\address[FF]{Department of Mathematical Sciences, Carnegie Mellon University, Pittsburgh, PA 15213, United States}
\email{frick@cmu.edu}
\thanks{The research of HA and FF was supported through the program ``Research in Pairs" by the Mathematisches Forschungsinstitut Oberwolfach in 2019.
The research of FF was supported in part by NSF grant DMS-1855591.}
\date{\today}

\begin{abstract}
\small
Thickenings of a metric space capture local geometric properties of the space.
Here we exhibit applications of lower bounding the topology of thickenings of the circle and more generally the sphere.
We explain interconnections with the geometry of circle actions on Euclidean space, the structure of zeros of trigonometric polynomials, and theorems of Borsuk--Ulam type.
We use the combinatorial and geometric structure of the convex hull of orbits of circle actions on Euclidean space to give geometric proofs of the homotopy type of metric thickenings of the circle.

Homotopical connectivity bounds of thickenings of the sphere allow us to prove that a weighted average of function values of odd maps $S^n \to \R^{n+2}$ on a small diameter set is zero.
We prove an additional generalization of the Borsuk--Ulam theorem for odd maps $S^{2n-1} \to \R^{2kn+2n-1}$.
We prove such results for odd maps from the circle to any Euclidean space with optimal quantitative bounds.
This in turn implies that any raked homogeneous trigonometric polynomial has a zero on a subset of the circle of a specific diameter; these results are optimal.
\end{abstract}

\maketitle

\section{Introduction}

A compact metric space $X$ admits a canonical isometric embedding into $C(X)^*$, the dual space of real-valued continuous functions on~$X$.
If $C(X)^*$ is equipped with a Wasserstein metric, then convex combinations of nearby points of $X$ in $C(X)^*$ give a canonical thickening of the space $X$ that exhibits local connectivity properties of~$X$.
In particular, if $X$ is a sufficiently dense sample of points in an ambient space $Y$, and $Y$ satisfies additional conditions such as being a closed Riemannian manifold with certain curvature bounds, then these metric thickenings of $X$ recover the homotopy type of $Y$ at small scale parameters~\cite{Latschev2001,AAF}.

In the present manuscript we relate metric thickenings of the circle $S^1$ (and more generally the $n$-sphere~$S^n$) to convexity properties of orbits of circle actions on Euclidean space, to Borsuk--Ulam type theorems, and to the structure of zeros of trigonometric polynomials with a prescribed spectrum.
We will briefly introduce these notions here and state our main results.

\subsection*{Borsuk--Ulam theorems for higher-dimensional codomains.}

The classical Borsuk--Ulam theorem states that any continuous map $f\colon S^n \to \R^n$ identifies some point with its antipode: $f(x) = f(-x)$ for some $x \in S^n$.
Equivalently, any odd map $f \colon S^n \to \R^n$, namely a map satisfying $f(-x) = -f(x)$ for all $x \in S^n$, must have a zero: $f(x) = \vec{0}$ for some $x \in S^n$.
For lower-dimensional codomains, Gromov's ``waist of the sphere" theorem gives quantitative bounds for size of the preimage of some point:
for any map $f \colon S^n \to \R^k$ with $k \le n$, there is a point $y \in \R^k$ such that $\varepsilon$-neighborhoods of $f^{-1}(y)$ have volume bounded below by the volume of the $\varepsilon$-neighborhood of an $(n-k)$-dimensional equator of~$S^n$~\cite{gromov2003isoperimetry,guth2008waist,memarian2011gromov}.
Here we investigate quantified generalizations of the Borsuk--Ulam theorem for maps to Euclidean space of dimension greater than~$n$; see~\cite{malyutin2018neighboring} for a different generalization.
While a generic odd map $f \colon S^n \to \R^k$ for $k > n$ does not have a zero, we will show that a convex combination of function values must achieve zero for points contained in a set of diameter strictly less than~$\pi$.
Further, the diameter bounds obtained are sharp for maps of the form $S^1\to \R^k$, $S^n \to \R^{n+1}$, and $S^n\to \R^{n+2}$.

In the following, $r_n$ denotes the diameter of a regular $(n+1)$-simplex inscribed into~$S^n$, where $S^n$ carries the standard spherical metric and where each great circle has length~$2\pi$.

\begin{theorem}\label{thm:bu}
If $f\colon S^1\to \R^{2k+1}$ is odd and continuous, then there is a subset $X\subseteq S^1$ of diameter at most $\frac{2\pi k}{2k+1}$ such that $\conv(f(X))$ contains the origin.
\end{theorem}

\begin{figure}[htb]
\centering
\includegraphics[width=0.15\textwidth]{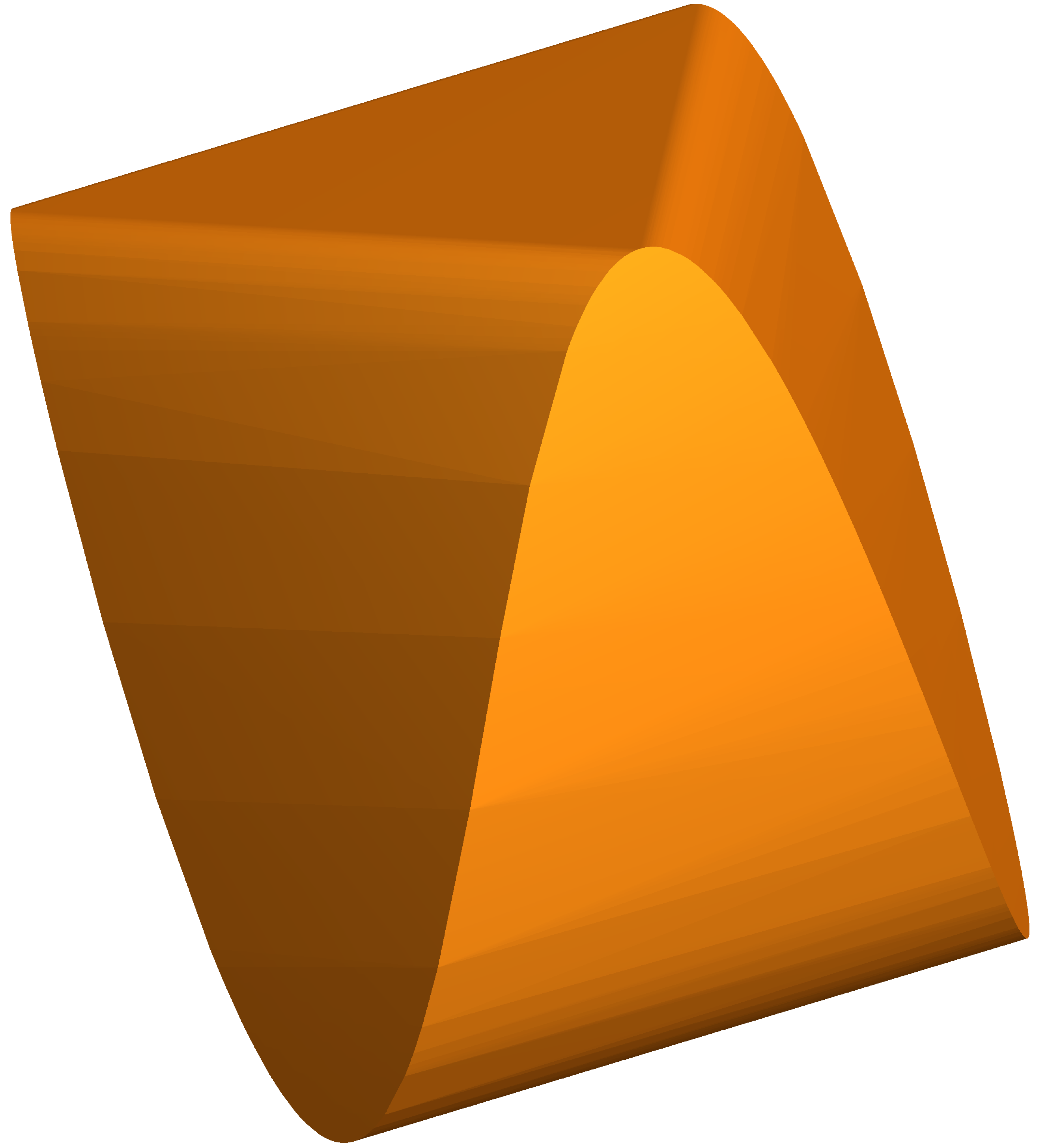}
\caption{The map $f\colon S^1\to \R^3$ defined by $f(t)=(\cos(t), \sin(t), \cos(3t))$ is odd; it contains (two) subsets $X\subseteq S^1$ of three equally-spaced points of diameter $\frac{2\pi}{3}$ with $\vec{0}\in\conv(f(X))$.}
\label{fig:convex}
\end{figure}

This result generalizes, with the same diameter bound, to odd maps $f\colon S^{2n-1}\to \R^{2kn+2n-1}$.
When $k=0$, we recover the classic Borsuk--Ulam theorem.

\begin{theorem}\label{thm:bu-Sn-2}
If $f\colon S^{2n-1}\to \R^{2kn+2n-1}$ is odd and continuous, then there is a subset $X\subseteq S^{2n-1}$ of diameter at most $\frac{2\pi k}{2k+1}$ such that $\conv(f(X))$ contains the origin.
\end{theorem}

We remark that the diameter bound obtained in Theorem~\ref{thm:bu-Sn-2} also applies to odd and continuous maps $f\colon S^{2n}\to \R^{2kn+2n-1}$; indeed, $f$ restricted to the equator $S^{2n-1}\hookrightarrow S^{2n}$ is odd and continuous with the same codomain.

\begin{theorem}\label{thm:bu-Sn}
If $f \colon S^n \to \R^{n+2}$ is odd and continuous, then there is a subset $X \subseteq S^n$ of diameter at most $r_n$ such that $\conv(f(X))$ contains the origin.
\end{theorem}

The diameter bounds in Theorems~\ref{thm:bu} and \ref{thm:bu-Sn} are optimal.
In the case of the circle, there exist odd maps $f\colon S^1 \to \R^{2k}\subseteq\R^{2k+1}$ already into one dimension lower such that $\conv(f(X))$ misses the origin for any set $X$ of diameter less than $\frac{2\pi k}{2k+1}$.
Constructing such an example map $f$ also proves a result about the structure of zeros of raked trigonometric polynomials, which we explain next.

\subsection*{The structure of zeros of raked trigonometric polynomials.}

A trigonometric polynomial is an expression of the form $p(t) = c +  \sum_{k=1}^n a_k \cos(kt) + b_k \sin(kt)$, inducing a map $S^1 \to \R$.
In the case that $c=0,$ we call $p$ a \emph{homogeneous} trigonometric polynomial.
The set $S \subseteq \{1,\dots,n\}$ of integers $k$ with $a_k \ne 0$ or $b_k \ne 0$ is called the \emph{spectrum} of~$p$, and the largest integer in $S$ is the \emph{degree} of~$p$.
The spectrum of $p$ constrains the set of roots of~$p$; for example, if $p$ is homogeneous of degree $n$ then it has a root on any closed circular arc of length~$\frac{2\pi n}{n+1}$; see~\cite{Babenko1984,GilbertSmyth2000}.
Kozma and Oravecz in~\cite{KozmaOravecz2002} give upper bounds on the length of an arc where a trigonometric polynomial with spectrum bounded away from zero (that is, $S \subseteq [k,n]$) is non-zero.
If the spectrum of $p$ consists only of odd integers, then $p$ is called a \emph{raked} trigonometric polynomial.
We show the following structural result about the roots of raked trigonometric polynomials:

\begin{theorem}
\label{thm:trig}
Let $X\subseteq S^1$ be such that $\diam(X)<\frac{2\pi k}{2k+1}$.
Then there is a raked homogeneous trigonometric polynomial of degree $2k-1$ that is positive on~$X$.
Moreover, there is a set $X \subseteq S^1$ of diameter $\frac{2\pi k}{2k+1}$ such that no raked homogeneous trigonometric polynomial of degree $2k-1$ is positive on~$X$.
\end{theorem}

The proof of the first part of Theorem~\ref{thm:trig} can be used to imply that the quantitative bound on the diameter in Theorem~\ref{thm:bu} is tight,
while the second part of this theorem is a corollary of Theorem~\ref{thm:bu}.

\subsection*{The symmetric moment curve and the Barvinok--Novik orbitope.}

The relation between Theorems~\ref{thm:bu} and~\ref{thm:trig} is explained by choosing $f\colon S^1 \to \R^{2k} \subseteq \R^{2k+1}$ to be the symmetric moment curve 
\[\sm_{2k}(t) = \bigl(\cos t, \sin t, \cos 3t, \sin 3t, \ldots, \cos(2k-1)t, \sin(2k-1)t\bigr),\]
which is an odd function.
The convex hull of the curve $\sm_{2k}$ is referred to as the \emph{Barvinok--Novik orbitope} $\cB_{2k}$.
Now, Theorem~\ref{thm:bu} implies that there is a set $X \subseteq S^1$ of diameter $\frac{2\pi k}{2k+1}$ such that $\conv(\sm_{2k}(X))$ captures the origin.
Thus, for any given $z \in \R^{2k}\setminus \{0\}$, the inner product $\langle z, \sm_{2k}(X)\rangle$ changes sign in~$X$ since no hyperplane can separate $\sm(X)$ from the origin.
For varying $z$, these inner products range over all possible raked homogeneous trigonometric polynomials, giving the second part of Theorem~\ref{thm:trig}.

The geometry of the Barvinok--Novik orbitope also shows that the bound $\frac{2\pi k}{2k+1}$ in Theorem~\ref{thm:bu} is optimal:

\begin{theorem}\label{thm:miss-origin}
Let $X\subseteq S^1$ be such that $\diam(X)<\mathcal{C}$.
Then the convex hull $\conv(\sm_{2k}(X))$ does not contain the origin $\vec{0}\in\R^{2k}$ if $\mathcal{C}=\frac{2\pi k}{2k+1}$, and this bound is sharp.
\end{theorem}

This also shows that, given $X\subseteq S^1$ with $\diam(X)<\frac{2\pi k}{2k+1}$, there is a raked trigonometric polynomial of degree $2k-1$ that is positive on~$X$.

\subsection*{Metric thickenings of the circle.}

For any metric space $X$, we get a continuum spectrum of metric spaces $\vrm{X}{r}$, $r > 0$, of \emph{Vietoris--Rips} metric thickenings that capture the local connectivity of~$X$~\cite{AAF}.
Note that the superscript $m$ is used here to denote the \emph{metric} thickening of the Vietoris--Rips complex, as opposed to the geometric realization, and is not a numerical parameter.
A point in $\vrm{X}{r}$ is a probability measure on $X$ with finite support of diameter at most $r$.
Recently, Vietoris--Rips thickenings have been used in topological data analysis and persistent homology~\cite{Carlsson2009,EdelsbrunnerHarer}; specifically, they allow for a growing filtration of topological spaces associated to a finite collection or sampling of data.
Conjecturally, for $X$ the circle~$S^1$, this spectrum ranges over all odd-dimensional spheres $S^1, S^3, S^5, \dots,$ until eventually becoming contractible.\footnote{These homotopy types are known for the Vietoris--Rips simplicial complexes $\vr{S^1}{r}$~\cite{adamaszek2017vietoris}, but not yet for the more natural Vietoris--Rips thickenings $\vrm{S^1}{r}$.}

\begin{conjecture}\label{conj:homotopy-S1}
For $\frac{2\pi(k-1)}{2k-1}\le r<\frac{2\pi k}{2k+1}$, the metric thickening $\vrm{S^1}{r}$ is homotopy equivalent to the boundary of the Barvinok--Novik orbitope $\cB_{2k}$, i.e.\ to the odd-dimensional sphere $S^{2k-1}$.
\end{conjecture}

As partial evidence towards this conjecture, we explain how Theorem~\ref{thm:miss-origin} implies that, for scale parameter $r$ in this range, the $(2k-1)$-dimensional homology, cohomology, and homotopy groups of $\vrm{S^1}{r}$ are nontrivial.

In Section~\ref{sec:results}, we show that Conjecture~\ref{conj:homotopy-S1} is true up to $r=\frac{2\pi}{3}$, the side-length of an inscribed equilateral triangle, where $\vrm{S^1}{r}\simeq S^3$.
More importantly, we provide a geometric picture of why this homotopy equivalence is plausible for $r\ge 0$, as follows.
For $\frac{2\pi(k-1)}{2k-1}\le r<\frac{2\pi k}{2k+1}$ we define a continuous map $\vrm{S^1}{r}\to\R^{2k}\setminus\{\vec{0}\}$ via the centrally symmetric moment curve $(\cos t, \sin t, \cos 3t, \sin 3t, \ldots)$.
We relate $\vrm{S^1}{r}$ to the facial structure of the Barvinok--Novik orbitope by composing with the radial projection map $\R^{2k}\setminus\{\vec{0}\}\to\partial\cB_{2k}$.
Finally, for $r=\frac{2\pi}{3}$ we obtain the homotopy equivalence $\vrm{S^1}{r}\simeq \partial \cB_4\cong S^3$ via a linear homotopy.
This is the only step that we are currently unable to extend to large $r$ and $k$; the missing ingredient is a ``diameter non-increasing" property for higher-dimensional Barvinok--Novik orbitopes (Conjecture~\ref{conj:diam}).

To our knowledge, this is the first approach to determine the homotopy type of a Vietoris--Rips thickening by mapping the underlying metric space into a higher-dimensional Euclidean space.
This technique is analogous to the ``kernel trick'' of machine learning, in which data is mapped into a higher dimensional space to illuminate the underlying structure of the data.

As a step towards understanding the relationship between metric thickenings of the circle and the Barvinok--Novik orbitopes, we show that given arbitrary $t_1,\dots, t_{2k-1}\in S^1$, there exists a raked homogeneous trigonometric polynomial $f$ of degree $2k-1$ with a root at each $t_i$ and its antipode. 
Further, the polynomial $f$ alternates signs between these roots, has no other roots in $S^1$, and may be written down explicitly in terms of the parameters $t_1,\dots, t_{2k-1}$ (Theorem~\ref{thm:farkasvector}).

We remark that there is an analogous connection between the \v{C}ech thickenings of the circle and the Carath\'{e}odory orbitopes, i.e.\ the convex hull of the curve $(\cos t, \sin t, \cos 2t, \sin 2t, \ldots, \cos kt, \sin kt\bigr)$~\cite{sanyal2011orbitopes}.

A preliminary version of several results in this paper appeared in the second author's master's thesis~\cite{BushMasters}.

\section{Preliminaries and related work}

In this section we review notation and related work on topology, Vietoris--Rips simplicial complexes, metric thickenings, convex geometry, moment curves, and orbitopes.

\subsection*{Topological and metric spaces}

We say two continuous maps $f,g\colon X\to Y$ are \emph{homotopic}, written $f\simeq g$, if there exists a continuous map $H\colon X\times [0,1]\to Y$ such that $H(x,0)=f(x)$ and $H(x,1)=g(x)$ for all $x\in X$~\cite{Hatcher}.
Such a map $H$ is called a \emph{homotopy}.
We say $X$ and $Y$ are \emph{homotopy equivalent}, denoted $X\simeq Y$, if there exist continuous maps $f\colon X\to Y$ and $g\colon Y\to X$ such that $g\circ f\simeq \mathrm{id}_X$ and $f\circ g \simeq \mathrm{id}_Y$.
We furthermore write $X\cong Y$ if spaces $X$ and $Y$ are homeomorphic.

Given a set of points $S\subseteq X$ in a metric space $(X,d)$, let the \emph{diameter} of $S$ be $\diam(S)=\sup\{d(x,y)~|~x,y\in S\}$; this value may be infinite.

\subsection*{Conventions regarding $S^1$}\label{ssec:circle}

We equip $S^1$ with the geodesic metric (of total circumference $2\pi$), though our results also hold when $S^1$ is instead equipped with the restriction of the Euclidean metric on $\R^2$.
Unless otherwise stated, we will always take a representative $t\in S^1=\R/2\pi\Z$ as belonging to $[0,2\pi)$.
Let $a,b\in S^1$, where $a\neq b$, and where $a$ and $b$ are each identified with a point in $[0,2\pi)$.
Define the \emph{open arc} $(a,b)_{S^1}$ as
\[(a,b)_{S^1}=\begin{cases}
\{t\in S^1 \mid a<t<b\}& \text{if }a<b\\
\{t\in S^1\mid a<t<b+2\pi \}& \text{if }a>b.
\end{cases}
\]
Define the \emph{closed arc} $[a,b]_{S^1}$ similarly. 

\subsection*{Vietoris--Rips simplicial complexes}

We identify an abstract simplicial complex with its geometric realization, which is a topological space.

\begin{definition}
Let $X$ be a metric space and fix $r\geq 0$.
The \emph{Vietoris--Rips simplicial complex of $X$ with scale parameter} $r$, denoted $\vr{X}{r}$, has $X$ as its vertex set and a finite subset $\sigma\subseteq X$ as a simplex whenever $\diam(\sigma)\leq r$.
\end{definition}
A point in $\vr{X}{r}$ can be written in barycentric coordinates as $\sum_{i=0}^k\lambda_i x_i$, with $\diam(\{x_0,\ldots,x_k\})\le r$.
We emphasize that in this paper we are using the $\le$ convention instead of the $<$ convention.

While the theorems of~\cite{hausmann1995vietoris,Latschev2001} describe conditions under which the homotopy type of a manifold is recoverable from a Vietoris--Rips complex for sufficiently small $r\ge 0$, much less is known about the topological behavior of these constructions for large values of $r$, even though large values of $r$ commonly arise in applications of persistent homology~\cite{EdelsbrunnerHarer}.
However, more is known in the specific case when the underlying manifold is the circle.
The following theorem from~\cite{adamaszek2017vietoris} is based on~\cite{Adamaszek2013,AAFPP-J}.

\begin{theorem}\label{thm:adamaszek2017vietoris}
Let $0\le r<\pi$.
There are homotopy equivalences 
\begin{align*}
\vr{S^1}{r}&\simeq\begin{cases} S^{2k-1} & \text{if }\, \frac{2\pi (k-1)}{2k-1}<r<\frac{2\pi k}{2k+1}\\
\bigvee^\mathfrak{c} S^{2k}& \text{if }\, r=\frac{2\pi k}{2k+1},
\end{cases}
\end{align*}
where $k=0,1,2,\ldots$, and where $\mathfrak{c}$ denotes the cardinality of the continuum.
\end{theorem}

Related papers include~\cite{gasparovic2018complete} which studies the 1-dimensional persistence of \v{C}ech and Vietoris--Rips complexes of metric graphs,~\cite{virk20171} which extends this to geodesic spaces,~\cite{virk2017approximations} which studies  approximations of Vietoris--Rips complexes by finite samples even at higher scale parameters, and~\cite{zaremsky2019} which applies Bestvina--Brady discrete Morse theory to Vietoris--Rips complexes.

\subsection*{Metric thickenings and optimal transport}

When a metric space $X$ is not finite, it is often impossible\footnote{A simplicial complex (for example $\vr{X}{r}$) is metrizable if and only if it is locally finite~\cite[Proposition~4.2.16(2)]{sakai2013geometric}.} to equip $\vr{X}{r}$ with a metric without changing the homeomorphism type.
In such instances the simplicial complex $\vr{X}{r}$ destroys the metric information about the underlying space $X$.
This motivates the consideration of the \emph{Vietoris--Rips metric thickening}, $\vrm{X}{r}$, which preserves metric information.

Let $\delta_x$ denote the Dirac delta mass at a point $x\in X$.

\begin{definition}[\cite{AAF}]\label{def:thickenings}
Let $X$ be a metric space and let $r\geq 0$.
The \emph{Vietoris--Rips thickening} is the set
\[\vrm{X}{r}=\left\{\sum_{i=0}^k \lambda_i \delta_{x_i}~\bigg\vert~k\in\N,\ x_i\in X,\ \diam(\{x_0,\dots,x_k\})\leq r,\ \lambda_i\geq 0,\ \sum \lambda_i=1\right\},\]
equipped with the $1$-Wasserstein metric.
\end{definition}

This metric is also called the Kantorovich, optimal transport, or earth mover's metric~\cite{vershik2013long,villani2003topics,villani2008optimal}; it provides a notion of distance between probability measures defined on a metric space.
Although it exists much more generally~\cite{edwards2011kantorovich,kellerer1984duality,kellerer1982duality}, the $1$-Wasserstein metric on $\vrm{X}{r}$ can be defined as follows.
Given $\mu,\mu'\in \vrm{X}{r}$ with $\mu=\sum_{i=0}^k \lambda_i \delta_{x_i}$ and $\mu'=\sum_{j=0}^{k'}\lambda_j' \delta_{x_j'}$, define a \emph{matching} $p$ between $\mu$ and $\mu'$ to be any collection of non-negative real numbers $\{p_{i,j}\}_{i,j}$ such that $\sum_{j=0}^{k'}p_{i,j}=\lambda_i$ and $\sum_{i=0}^k p_{i,j}=\lambda_j'$.
Define the \emph{cost} of the matching $p$ to be $\sum_{i,j} p_{i,j}d(x_i, x_j')$.
The $1$-Wasserstein distance between $\mu,\mu'\in\vrm{X}{r}$ is then the infimum, varying over all matchings $p$ between $\mu$ and $\mu'$, of the cost of $p$.

Note that $\vrm{X}{0}$ is isometric to $X$.
Contrary to the situation for an arbitrary Vietoris--Rips complex, the embedding $X\to \vrm{X}{r}$ into the Vietoris--Rips metric thickening given by $x\mapsto \delta_x$ is continuous.
In fact, more is true: $\vrm{X}{r}$ is an \emph{$r$-thickening} of $X$~\cite{Gromov,AAF}.
For this reason, we identify $x\in X$ with the measure $\delta_x\in \vrm{X}{r}$ in the image of this embedding. 

If $M$ is a complete Riemannian manifold with curvature bounded from above and below, then $\vrm{M}{r}$ is homotopy equivalent to $M$ for $r$ sufficiently small~\cite{AAF,AM}.
This property provides an analogue of Hausmann's theorem~\cite{hausmann1995vietoris} for metric thickenings.

Given a measure $\mu=\sum_{i=0}^k\lambda_i \delta_{x_i}$ with $\lambda_i>0$, we denote the \emph{support} of $\mu$ by $\supp(\mu)=\{x_0,\ldots,x_k\}$.

\subsection*{Convex geometry}

Convex geometry is the study of convex sets, especially polytopes and their facial structures~\cite{ziegler2012lectures}.
Given an arbitrary subset $Y\subseteq\R^n$, we let
\[\conv(Y)=\left\{\sum_{i=1}^k \lambda_i v_i~\bigg\vert~k\in\N,\ v_i\in Y,\ \lambda_i \geq 0,\ \sum_{i=1}^k \lambda_i = 1\right\}\]
denote the \emph{convex hull} of $Y$.
For example, Figure~\ref{fig:convex} shows the convex hull of the image of the map $f\colon S^1\to\R^3$ defined by $f(t) = (\cos(t), \sin(t), \cos(3t))$.
Similarly, the \emph{conical hull} of $Y$ is
\[\cone(Y)=\left\{\sum_{i=1}^k \lambda_i v_i~\bigg\vert~k\in\N,\ v_i\in Y,\ \lambda_i \geq 0\right\}.\]

Let $Y\subseteq \R^n$ be convex.
Define a \emph{face} of $Y$ to be any convex set $F\subseteq Y$ such that, given $x\in F$, if $x=\lambda y +(1-\lambda)z$ for some $0<\lambda<1$ and $y,z\in Y$, then $y,z\in F$.

\subsection*{The centrally symmetric trigonometric moment curve}

The centrally symmetric moment curve is analogous to the trigonometric moment curve, with the additional property that it is symmetric under reflecting through the origin.

\begin{definition}\label{def:sm-2k}
For $k\in\N$, the \emph{centrally symmetric moment curve} $\sm_{2k}\colon S^1\to\R^{2k}$ is defined by
\[\sm_{2k}(t) = \bigl(\cos t, \sin t, \cos 3t, \sin 3t, \ldots, \cos(2k-1)t, \sin(2k-1)t\bigr).\]
\end{definition}

Here we identify the domain $S^1$ with $\R/2\pi\Z$.
Since $\sm_{2k}(t+\pi)=-\sm_{2k}(t)$, we say that $\sm_{2k}$ is \emph{centrally symmetric} about the origin.
Interestingly, this curve is closely related to the multidimensional scaling (MDS) embedding $S^1\hookrightarrow \R^{2k}$ of the geodesic circle~\cite{adamsBlumsteinKassab,bibby1979multivariate,kassab2019multidimensional,von1941fourier}; multidimensional scaling is a way to map a metric space into Euclidean space in a way that distorts the metric (in some sense) as little as possible.

\subsection*{Barvinok--Novik orbitopes}

The \emph{Barvinok--Novik orbitope} is defined by $\cB_{2k} = \conv(\sm_{2k}(S^1)) \subseteq \R^{2k}$~\cite{barvinok2008centrally}.
This convex body is not the convex hull of a finite set of points; it is an \emph{orbitope} instead of a polytope~\cite{sanyal2011orbitopes}.

The faces of $\cB_{2k}$ are known for $k=2$; a subset of these faces are visible in Figure~\ref{fig:convex} (which is in $\R^3$ instead of $\R^4$).

\begin{theorem}[\cite{barvinok2008centrally,smilansky1985convex}]\label{thm:faces}
The proper faces of $\cB_4$ are 
\begin{itemize}
\item the 0-dimensional faces (vertices) $\sm_4(t)$ for $t\in S^1$,
\item the 1-dimensional faces (edges) $\conv(\sm_4(\{t_1,t_2\}))$ where $t_1\neq t_2$ are the edges of an arc of $S^1$ of length at most $\frac{2\pi}{3}$, and
\item the 2-dimensional faces (triangles) $\conv(\sm_4(\{t,t+\frac{2\pi}{3},t+\frac{4\pi}{3}\}))$ for $t\in S^1$.
\end{itemize}
\end{theorem}

Though the facial structure of the Barvinok--Novik orbitopes $\cB_{2k}$ is not known for $k>2$, certain neighborliness results have been established~\cite{barvinok2013neighborliness}.
Sinn has shown that the orbitopes are simplicial~\cite{sinn2013algebraic}. 
Additionally, Vinzant proved that the edges of $\partial\cB_{2k}$ consist of all line segments $\conv\left(\sm_{2k}(\{t_0,t_1\})\right)$ with $|t_0-t_1|\le\frac{2\pi(k-1)}{2k-1}$~\cite{vinzant2011edges}.
In other words, the edges of $\cB_{2k}$ are the same as the edges of $\vr{S^1}{\frac{2\pi (k-1)}{2k-1}}$.
The following is an immediate corollary of the work of Sinn and Vinzant.

\begin{corollary}[\cite{sinn2013algebraic,vinzant2011edges}]\label{cor:inclusion}
Every face of the Barvinok--Novik orbitope $\cB_{2k}$ is a simplex whose diameter in $S^1$ (not in $\R^{2k}$) is at most $\frac{2\pi (k-1)}{2k-1}$.
\end{corollary}

\section{A generalization of the Borsuk--Ulam theorem}\label{sec:bu}

The Borsuk--Ulam states that if $f\colon S^n\to \R^n$ is continuous, then there exists a point $x\in S^n$ with $f(x)=f(-x)$~\cite{matousek2003using}.
For maps into lower-dimensional Euclidean space, there is a generalization due to Gromov called the ``waist of the sphere" theorem~\cite{gromov2003isoperimetry,guth2008waist,memarian2011gromov}.
The theorem says that if $f\colon S^n\to\R^k$ is continuous with $k\le n$, then there is some point $y\in \R^k$ such that the $\varepsilon$-neighborhoods of $f^{-1}(y)$ have volume at least as large as the volume of the $\varepsilon$-neighborhood of an $(n-k)$-dimensional equator of~$S^n$.
There are also versions in which the size of a preimage is measured by its diameter; this is called the Urysohn width~\cite{akopyan2012borsuk,maliszewski2014level,tikhomirov1971some}.
In this section, we ask: what can be said for maps $f\colon S^n\to\R^k$ with $k\ge n$?

We say a map $f\colon S^n\to \R^k$ is \emph{odd} or \emph{centrally-symmetric} if $f(-x)=-f(x)$ for all $x\in S^n$.
An equivalent formulation of the Borsuk--Ulam states that if $f\colon S^n\to \R^n$ is continuous and odd, then there exists a point $x\in S^n$ with $f(x)=\vec{0}$~\cite[Theorem~2.1.1]{matousek2003using}.

More generally, given topological spaces $X$ and $Y$ equipped with $\Z/2\Z$-actions $\mu$ and $\nu$ respectively, we say a map $f\colon X\to Y$ is \emph{odd} or  \emph{$\Z/2\Z$-equivariant} if $f\circ\mu=\nu\circ f$.
Additionally, we always equip $\R^n$ with the standard antipodal $\Z/2\Z$-action specified by $x\mapsto -x$.
For our purposes, we consider another useful reformulation of the Borsuk--Ulam theorem as follows.
Given an $(n-1)$-connected space $X$,~\cite[Proposition~5.3.2(iv)]{matousek2003using} states that there is no $\Z/2\Z$-equivariant map from $X$ into $S^{n-1}$.
Hence, any odd map $f\colon X\to\R^n$ must hit the origin, because otherwise we would obtain an odd map $\frac{f}{|f|}\colon X\to S^{n-1}$.

In Theorems~\ref{thm:bu},~\ref{thm:bu-Sn-2}, and~\ref{thm:bu-Sn} we prove the generalizations of the Borsuk--Ulam thoerem for maps $S^n\to\R^k$ with $k\ge n$.
The first result is for $n=1$.

\begin{theorem-bu}
If $f\colon S^1\to \R^{2k+1}$ is odd and continuous, then there is a subset $X\subseteq S^1$ of diameter at most $\frac{2\pi k}{2k+1}$ such that $\conv(f(X))$ contains the origin.
\end{theorem-bu}

Equivalently, if $f\colon S^1\to \R^{2k+1}$ is continuous, then there exists a subset $\{x_1,\ldots,x_m\}\subseteq S^1$ of diameter at most $\frac{2\pi k}{2k+1}$ such that $\sum_{i=1}^m\lambda_i f(x_i)=\sum_{i=1}^m\lambda_i f(-x_i)$, for some choice of convex coefficients $\lambda_i$.

For example, if 
$f=\sm_{2k}\colon S^1 \to \R^{2k}\subseteq \R^{2k+1}$, then this set $X$ is easy to find: we can let $X$ be $2k+1$ evenly-spaced points on the circle.
Theorem~\ref{thm:miss-origin} shows that the above diameter bound is sharp, both for maps $S^1\to \R^{2k+1}$ and for maps $S^1\to\R^{2k}$.
Indeed, $\sm_{2k}\colon S^1 \to \R^{2k}\subseteq \R^{2k+1}$ is an odd map in which the convex hull of the image of every set of diameter strictly less than $\frac{2\pi k}{2k+1}$ misses the origin.

\begin{proof}
Fix $0<\varepsilon<\frac{2\pi}{(2k+1)(2k+3)}$; this ensures $\frac{2\pi k}{2k+1}+\varepsilon<\frac{2\pi(k+1)}{2k+3}$.
The induced map $f\colon\vr{S^1}{\frac{2\pi k}{2k+1}+\varepsilon}\to\R^{2k+1}$ given by $f(\sum_i\lambda_i x_i)=\sum_i\lambda_i f(x_i)$ is odd with domain $\vr{S^1}{\frac{k}{2k+1}+\varepsilon}\simeq S^{2k+1}$ by Theorem~\ref{thm:adamaszek2017vietoris}.
By the Borsuk--Ulam theorem, this map has a zero, giving a subset $X$ of diameter at most $\frac{2\pi k}{k+1}+\varepsilon$ with $\conv(f(X))$ containing the origin.
Furthermore, by Carath\'{e}odory's theorem, we can take the size of $X$ to be at most $2k+2$.
It remains to reduce the diameter to $\frac{2\pi k}{k+1}$ by a compactness argument.

For each integer $n\ge 1$, we obtain a subset $X_n\subseteq S^1$ of diameter bounded above by $\frac{2\pi k}{2k+1}+\frac{\varepsilon}{n}$, of size $|X_n|\le 2k+2$, and with $\vec{0}\in\conv(f(X_n))$.
If $|X_n|<2k+2$, then duplicate an arbitrary point in $X_n$ to obtain a multi-set of size exactly $2k+2$.
Arbitrarily order these points so that $X_n$ can be thought of as a point in the torus $(S^1)^{2k+2}$.
By compactness of the torus, the sequence $\{X_n\}$ has a subsequence converging to a limit configuration $X\in(S^1)^{2k+2}$ of diameter at most $\frac{2\pi k}{2k+1}$ and with $\vec{0}\in\conv(f(X))$.
Removing duplicate points (and ignoring the ordering) gives us the desired subset $X\subseteq S^1$.
\end{proof}

We remark that Theorem~\ref{thm:bu} may also be proven using the Barvinok--Novik orbitopes and Corollary~\ref{cor:inclusion}, though this perspective does not generalize as nicely for the proofs of Theorems~\ref{thm:bu-Sn-2} and~\ref{thm:bu-Sn}.
In particular, an odd and continuous map $f\colon S^1\to \R^{2k+1}$ defines a continuous map $F\colon \vrm{S^1}{\frac{2\pi k}{2k+1}}\to\R^{2k+1}$ obtained by extending linearly to convex combinations of points in $S^1$. 
Then, $F$ induces an odd and continuous map $\tilde{F}\colon \partial\cB_{2k+2}\to \R^{2k+1}$ factoring through the well-defined inclusion $\iota\colon\partial\cB_{2k+2}\hookrightarrow{} \vrm{S^1}{\frac{2\pi k}{2k+1}}$ (this inclusion exists by Corollary~\ref{cor:inclusion}).
Since $\partial\cB_{2k+2}$ is homeomorphic to a $(2k+1)$-sphere and hence $2k$-connected, we may apply~\cite[Proposition~5.3.2(iv)]{matousek2003using} to obtain the result.
This approach depends crucially on known properties of the boundary of the Barvinok--Novik orbitope $\cB_{2k+2}$ and does not immediately generalize to maps from higher-dimensional spheres.


\begin{corollary}\label{cor:odd-functions}
Fix a list of odd continuous functions $f_i(t)\colon S^1 \to \R$ for $1\le i\le 2k+1$.
Let $P$ be the set of functions of the form $p\colon S^1\to \R$ defined by $p(t)=\sum_{j=1}^{2k+1}z_j f_j(t)$ with $z_j\in \R$.
Then there is a subset $X\subseteq S^1$ of diameter at most $\frac{2\pi k}{2k+1}$ such that no function in $P$ is positive on $X$. 
\end{corollary}

\begin{proof}
Consider the odd map $f\colon S^1\to\R^{2k+1}$ given by $f(t)=(f_1(t),\ldots,f_{2k+1}(t))$.
Note that each function $p\in P$ is specified by a vector $z\in\mathbb{R}^{2k+1}$, in the sense that $p(t)=z^\intercal f(t)$ for all $t\in S^1$.
By Theorem~\ref{thm:bu}, there exists a subset $X\subseteq S^1$ of diameter at most $\frac{2\pi k}{2k+1}$ such that $\conv(f(X))$ contains the origin.
Hence, if we write $X=\{x_1,\ldots,x_m\}$ with $\sum_{i=1}^m\lambda_i f(x_i)=\vec{0}$ for $\lambda_i\ge 0$, then $\sum_{i=1}^m\lambda_i p(x_i)=\sum_{i=1}^m \lambda_i z^\intercal f(x_i)=z^\intercal\sum_{i=1}^m\lambda_i f(x_i)=z^\intercal\vec{0}=0$. 
In particular, $p(x_i)$ must be non-positive for at least some~$i$.
\end{proof}

The next corollary follows immediately from Corollary~\ref{cor:odd-functions}, and proves the second part of Theorem~\ref{thm:trig}.

\begin{corollary}\label{cor:P}
Fix a list of odd degrees $d_i$ for $1\le i\le 2k+1$, and fix a list of trigonometric functions $f_i(t)=\sin(t)$ or $f_i(t)=\cos(t)$.
Let $P$ be the set of all polynomials of the form $p(t)=\sum_{j=1}^{2k+1}z_j f_j(d_j t)$ with $z_j\in \R$.
Then there is a subset $X\subseteq S^1$ of diameter at most $\frac{2\pi k}{2k+1}$ such that no polynomial in $P$ is positive on $X$.
\end{corollary}

For example, the above corollary applies if $P$ is the set of all raked homogeneous trigonometric polynomials of degree at most $2k-1$, namely
\[p(t)=\sum_{j=1}^k a_j \cos\bigl((2j-1)t\bigr)+\sum_{j=1}^k b_j \sin\bigl((2j-1)t\bigr),\]
after noting that we are considering the special case in which one of the constants $z_j$ defining $p(t)=\sum_{j=1}^{2k+1}z_j f_j(d_j t)$ is zero.

We also show the sharpness of the above result: the number of summands defining the trigonometric polynomial cannot be increased, and the upper bound on the diameter of $X$ can not be decreased, giving the first part of Theorem~\ref{thm:trig}.

\begin{corollary}\label{cor:pos-trig-poly}
Given a subset $X\subseteq S^1$ of diameter less than $\frac{2\pi k}{2k+1}$, there exists a raked homogeneous trigonometric polynomial of degree $2k-1$ that is positive on all of the points in $X$.
\end{corollary}

\begin{proof}
This result is a corollary of Theorem~\ref{thm:miss-origin} in Section~\ref{sec:miss-origin}, which says that the convex hull $\conv(\sm_{2k}(X))$ does not contain the origin.
Hence, there is a separating hyperplane $H_z$ with orthogonal vector $z\in \R^{2k}$ given by $H_z=\{x\in\R^{2k}~|~z^\intercal x > 0\}$ such that $\sm_{2k}(X)\subseteq H_z$.
Therefore, the raked homogeneous trigonometric polynomial of degree $2k-1$ given by $p_z(x)=z^\intercal \sm_{2k}(x)$ is positive on all of the points in $X$.
\end{proof}

We give two versions of the Borsuk--Ulam theorem for maps $S^n\to\R^k$ with $k\ge n$, now also with $n\ge 2$.
Our results would be strengthened if we better understood the homotopy types of Vietoris--Rips thickenings of $n$-spheres for $n\ge2$ at all scale parameters.

The first version generalizes Theorem~\ref{thm:bu} (take $n=1$) to maps from odd-dimensional spheres into Euclidean spaces of higher dimensions.

\begin{theorem-bu-Sn-2}
If $f\colon S^{2n-1}\to \R^{2kn+2n-1}$ is odd and continuous, then there is a subset $X\subseteq S^{2n-1}$ of diameter at most $\frac{2\pi k}{2k+1}$ such that $\conv(f(X))$ contains the origin.
\end{theorem-bu-Sn-2}

\begin{proof}
The case of $k=0$ follows from the standard Borsuk--Ulam theorem.

For $k\ge 1$, we will think of $S^{2n-1}$ as a join of $n$ circles $(S^1)^{*n}$.
Explicitly, if $S^{2n-1}$ is viewed as the unit sphere in $\R^{2n}$, then the subset of $S^{2n-1}$ with all coordinates zero, with the (possible) exception of coordinates $2i-1$ and $2i$, is a circle.
The distance between any two points in distinct such circles is $\frac{\pi}{2}$ in the geodesic metric.
Let $\frac{2\pi k}{2k+1}<r<\frac{2\pi (k+1)}{2k+3}$.
Since $k\ge 1$ implies $r>\frac{\pi}{2}$, this will allow us to construct a $\Z/2\Z$-equivariant embedding of $(\vr{S^1}{r})^{*n}$ into $\vr{S^{2n-1}}{r}$.
In barycentric coordinates, a point in $\vr{S^1}{r}$ can be written as $\sum_{x\in X}\lambda_x x$, where the vertex set $X$ of the simplex containing this point has diameter at most $r$ and where the positive coefficients $\lambda_x$ sum to one.
Hence, a point in $(\vr{S^1}{r})^{*n}$ consists of $n$ collections of such points $\sum_{x\in X_i}\lambda_x x$ for $1\le i\le n$, along with non-negative numbers $\kappa_1, \ldots, \kappa_n$ that add up to one.
We map the points in $X_i$ to the $i$-th copy of $S^1$ in $S^{2n-1} = (S^1)^{*n}$, and we multiply their weights
by $\kappa_i$.
This gives the barycentric coordinates of a well-defined point in $\vrm{S^{2n-1}}{r}$; the diameter of the supporting simplex is at most $r$ since $r>\frac{\pi}{2}$.
Furthermore, this map respects the antipodal $\Z/2\Z$-actions on $(\vr{S^1}{r})^{*n}$ and $\vr{S^{2n-1}}{r}$; these actions are free since antipodal points are at distance $r<\pi$ apart.
By Theorem~\ref{thm:adamaszek2017vietoris} we have
\[(\vr{S^1}{r})^{*n}\simeq (S^{2k+1})^{*n}=S^{(2k+1)n+n-1}=S^{2kn+2n-1}.\]
It follows from~\cite[Proposition~5.3.2(iv)]{matousek2003using} that any odd map from $(\vr{S^1}{r})^{*n}$, and hence also from $\vr{S^{2n-1}}{r}$, into $\R^{2kn+2n-1}$ hits the origin.
This gives a subset $X\subseteq S^{2n-1}$ of diameter at most $r=\frac{2\pi k}{2k+1}+\varepsilon$ such that $\conv(f(X))$ contains the origin.
By a compactness argument as in the proof of Theorem~\ref{thm:bu}, we can reduce this diameter to exactly $\frac{2\pi k}{2k+1}$.
\end{proof}

In the following theorem, $r_n$ is the diameter of an inscribed regular $(n+1)$-simplex in~$S^n$.

\begin{theorem-bu-Sn}
If $f \colon S^n \to \R^{n+2}$ is odd and continuous, then there is a subset $X \subseteq S^n$ of diameter at most $r_n$ such that $\conv(f(X))$ contains the origin.
\end{theorem-bu-Sn}

This diameter bound is sharp, both for maps $S^n\to\R^{n+2}$ and maps $S^n\to\R^{n+1}$.
Indeed, the standard inclusion $f\colon S^n\hookrightarrow \R^{n+1}\subseteq \R^{n+2}$ is an odd map that satisfies $\vec{0}\notin\conv(f(X))$ for all $X\subseteq S^n$ of diameter less than $r_n$~\cite[Proof of Lemma~3]{lovasz1983self}.

\begin{proof}
The space $\vrm{S^n}{r_n}$ has a free $\Z/2\Z$-action that maps the convex combination $\sum_{i=1}^k \lambda_i\delta_{x_i}$ of Dirac measures for points $x_1, \dots, x_k$ on~$S^n$ to $\sum_{i=1}^k \lambda_i \delta_{-x_i}$, that is, to the measure that is supported on the antipodal point set with the same weights~$\lambda_i$.
This action is free since antipodal points on $S^n$ are farther than $r_n$ apart. 

Let $f\colon S^n\to\R^{n+2}$ be odd and continuous. Because $f$ is bounded,~\cite[Lemma~5.2]{AAF} implies that $f$ induces a continuous map $F\colon \vrm{S^n}{r_n} \to \R^{n+2}$ defined by $F(\sum_{i=1}^k \lambda_i\delta_{x_i}) = \sum_{i=1}^k \lambda_if(x_i)$. Notice that $F$ commutes with the antipodal action on $\vrm{S^n}{r_n}$ and~$S^n$:
\[F\left(\sum_{i=1}^k \lambda_i\delta_{-x_i}\right) = \sum_{i=1}^k \lambda_if(-x_i) = -\sum_{i=1}^k \lambda_if(x_i) = -F\left(\sum_{i=1}^k \lambda_i\delta_{x_i}\right).\]

Next, fix a regular $(n+1)$-simplex $\Delta$ inscribed in $S^n$, and let $A_{n+2}$ denote the group of rotational symmetries of $\Delta$, that is, the alternating group on $n+2$ elements.
In a noncanonical fashion, we may identify $A_{n+2}$ as a subgroup of $\mathrm{SO}(n+1)$ by associating to each $g\in A_{n+2}$ the matrix $M_g\in \mathrm{SO}(n+1)$ such that $M_g\cdot v=g\cdot v$ for each vertex $v$ of $\Delta$.
In this way, we obtain the orbit space $\frac{\mathrm{SO}(n+1)}{A_{n+2}}$ of $\mathrm{SO}(n+1)$ under the action of $A_{n+2}$ by left multiplication.
Theorem~5.4 of~\cite{AAF} states that the homotopy type of $\vrm{S^n}{r_n}$ is $\Sigma^{n+1} \frac{\mathrm{SO}(n+1)}{A_{n+2}}$, and because $\frac{\mathrm{SO}(n+1)}{A_{n+2}}$ is connected, its $(n+1)$-fold suspension $\Sigma^{n+1} \frac{\mathrm{SO}(n+1)}{A_{n+2}}\simeq\vrm{S^n}{r_n}$ is $(n+1)$-connected.
Thus, the map~$F$, as a $\Z/2\Z$-equivariant map from an $(n+1)$-connected space to~$\R^{n+2}$, has a zero~\cite[Proposition~5.3.2(iv)]{matousek2003using}.
That is, there are points $x_1, \dots, x_m \in S^n$ that are pairwise at distance at most $r_n$ and such that $\sum_{i=1}^m \lambda_i f(x_i) =\vec{0}$ for some $\lambda_1, \dots, \lambda_m \ge 0$ with $\sum_{i=1}^m \lambda_i = 1$.
\end{proof}

\section{Diameter bound for Carath{\'e}odory sets on the symmetric moment curve}\label{sec:miss-origin}

Let $Y\subseteq \R^k$ be a set in Euclidean space.
Carath{\'e}odory's theorem states that if the convex hull of $Y$ contains the origin, then there is a subset of $Y$ of at most $k+1$ points whose convex hull also contains the origin. 
We say that $Y'\subseteq Y$ is a \emph{Carath{\'e}odory subset of $Y$} if the convex hull of $Y'$ contains the origin.
The following theorem gives a lower bound on the diameter of the preimage of any Carath{\'e}odory subset of the symmetric moment curve in $\R^{2k}$.
Here, the circle $S^1$ is equipped with the geodesic metric of total circumference $2\pi$.

\begin{theorem-miss-origin}
Let $X\subseteq S^1$ be such that $\diam(X)<\mathcal{C}$.
Then the convex hull $\conv(\sm_{2k}(X))$ does not contain the origin $\vec{0}\in\R^{2k}$ if $\mathcal{C}=\frac{2\pi k}{2k+1}$, and this bound is sharp.
\end{theorem-miss-origin}

To prove Theorem~\ref{thm:miss-origin}, we may restrict attention to subsets of $\sm_{2k}(S^1)$ of size at most $2k+1$ by Carath{\'e}odory's theorem.
Suppose $X=\{t_0,\dots,t_{2k}\}\subseteq S^1$ is such that the origin is contained in the convex hull of $\{\sm_{{2k}}(t_0),\dots,\sm_{2k}(t_{2k})\}$.
Then, there exist scalars $\lambda_i\geq 0$ such that $\vec{0}=\sum_{i=0}^{2k} \lambda_i \sm_{2k}(t_i)$ and $\sum_{i=0}^{2k} \lambda_i=1$.
In this way, we obtain a system of $2k$ equations
\[\sum_{i=0}^{2k} \lambda_i \cos(n t_i)=0 \quad\text{and}\quad \sum_{i=0}^{2k} \lambda_i \sin(n t_i)=0 \quad\text{for}\quad n=1,3,\dots,2k-1.\]
We therefore let $M_{2k}$ be the $2k\times(2k+1)$ matrix
\[M_{2k}=\begin{pmatrix} 
\cos(t_0) & \cos(t_1)   & \hdots & \cos(t_{2k})  \\
\sin(t_0) & \sin(t_1)   & \hdots & \sin(t_{2k})  \\
\cos(3t_0) & \cos(3t_1) & \hdots & \cos(3t_{2k}) \\
\sin(3t_0) & \sin(3t_1) & \hdots & \sin(3t_{2k}) \\
\vdots & \vdots & \ddots & \vdots \\
\cos((2k-1)t_0) & \cos((2k-1)t_1) & \hdots & \cos((2k-1)t_{2k}) \\
\sin((2k-1)t_0) & \sin((2k-1)t_1) & \hdots & \sin((2k-1)t_{2k}) \\
\end{pmatrix},\]
and consider the vector equation $M_{2k}\vec{\lambda}=\vec{0}$.
To prove Theorem~\ref{thm:miss-origin}, we build towards describing the nullspace of $M_{2k}$, which we complete in Lemma~\ref{lem:nullspace}.

\begin{lemma}\label{lem:det}
Let $A$ denote the $2k\times 2k$ matrix whose columns are $\sm_{2k}(t_1),\sm_{2k}(t_2),\ldots,\sm_{2k}(t_{2k})$.
Then
$\det(A)=\kappa \prod_{1\leq j<l\leq 2k}\sin(t_l-t_j)$
for some nonzero constant $\kappa$ depending only on $k$.
\end{lemma}

We would like to thank Harrison Chapman for the insights behind the proof of Lemma~\ref{lem:det}.
The main idea of the proof is to perform elementary row and column operations to $A$ to obtain a Vandermonde matrix.
In addition to the general case, the simpler case $k=2$ of this proof is written out in more detail in~\cite{BushMasters}.
The determinant of a related matrix is given in~\cite{gallier2003computing}.

A \emph{Vandermonde matrix} is an $n\times n$ matrix of the form 
\[V=\begin{pmatrix}
1 & a_1 & a_1^2&\cdots &a_1^{n-1}\\
1 & a_2 & a_2^2&\cdots &a_2^{n-1}\\
\vdots & \vdots & \vdots &\ddots &\vdots \\
1 & a_n & a_n^2&\cdots &a_n^{n-1}\\
\end{pmatrix}.\]
Its determinant is $\det(V)=\prod_{1\leq i< j\leq n}(a_j-a_i)$; see for example~\cite[Section~2.8.1]{press2007numerical}.

\begin{proof}[Proof of Lemma~\ref{lem:det}]

We will perform elementary row and column operations to $A$ to obtain a Vandermonde matrix.
Given $f\colon \R\to\mathbb{C}$, define the function $f\colon\R^{2k}\to\mathbb{C}^{2k}$ via
$f(\ut)= ( f(t_1), f(t_2), \ldots, f(t_{2k}) )^{\text{T}}$ for $\ut=(t_1,t_2,\ldots,t_{2k})^{\text{T}}$.
Since
\[A=\begin{pmatrix} 
\cos(t_1) & \cos(t_2)   & \hdots & \cos(t_{2k})  \\
\sin(t_1) & \sin(t_2)   & \hdots & \sin(t_{2k})  \\
\cos(3t_1) & \cos(3t_2) & \hdots & \cos(3t_{2k}) \\
\sin(3t_1) & \sin(3t_2) & \hdots & \sin(3t_{2k}) \\
\vdots & \vdots & \ddots & \vdots \\
\cos((2k-1)t_1) & \cos((2k-1)t_2) & \hdots & \cos((2k-1)t_{2k}) \\
\sin((2k-1)t_1) & \sin((2k-1)t_2) & \hdots & \sin((2k-1)t_{2k}) \\
\end{pmatrix},\]
we have
\begin{align*}
\det(A)&=\det\left(A^{\text{T}}\right) =\det\begin{pmatrix} \cos(\ut) &  \sin(\ut) & \cos(3\ut) &\sin(3\ut) &\cdots & \cos((2k-1)\ut)& \sin((2k-1)\ut)\end{pmatrix}\\
& =\det \begin{pmatrix} \frac{e^{i\ut}+e^{-i\ut}}{2} &  \frac{e^{i\ut}-e^{-i\ut}}{2i} &  \frac{e^{3i\ut}+e^{-3i\ut}}{2} & \frac{e^{3i\ut}-e^{-3i\ut}}{2i} & \cdots & \frac{e^{(2k-1)i\ut}+e^{-(2k-1)i\ut}}{2} & \frac{e^{(2k-1)i\ut}-e^{-(2k-1)i\ut}}{2i} \end{pmatrix} \\
&=\frac{1}{2^{2k}}(-i)^k \det \begin{pmatrix} e^{i\ut}+e^{-i\ut} &  e^{i\ut}-e^{-i\ut} & \cdots &  e^{(2k-1)i\ut}+e^{-(2k-1)i\ut} & e^{(2k-1)i\ut}-e^{-(2k-1)i\ut}  \end{pmatrix}.
\end{align*}
Next, let $C_j$ denote the $j$-th column of the above matrix.
For $j=1,3,\dots,2k-1$, perform the column operations $C_j\mapsto C_j+C_{j+1}$, and then after each $C_j$ has been updated, perform the column operations $C_{j+1}\mapsto C_{j+1}-\frac{1}{2}C_j$.
It follows that
\begin{align*} \det(A)& =\frac{1}{2^{2k}}(-i)^k  \det \begin{pmatrix} 2e^{i\ut} & -e^{-i\ut} & 2e^{3 i\ut} & -e^{-3i\ut} & \cdots & 2e^{(2k-1)i\ut} & -e^{-(2k-1)i \ut} \end{pmatrix}\\
& =\frac{i^k}{2^{k}} \det \begin{pmatrix} e^{i\ut} & e^{-i\ut} & e^{3 i\ut} & e^{-3i\ut}& \cdots & e^{(2k-1)i\ut} & e^{-(2k-1)i \ut}  \end{pmatrix}
\end{align*}
by factoring out column multiples.
Letting $\omega=e^{-(2k-1)i (t_1+t_2+\dots+t_{2k})}$, we may factor $e^{-(2k-1)i t_j}$ from row $j$ to obtain 
\begin{align*} \det(A)&  = \frac{i^k}{2^{k}}\omega \det \begin{pmatrix} e^{((2k-1)+1)i\ut} & e^{((2k-1)-1)i\ut} & \cdots &  e^{((2k-1)+(2k-1))i\ut} &  e^{((2k-1)-(2k-1))i\ut} \end{pmatrix}\\
&=\frac{i^k}{2^{k}}\omega \det \begin{pmatrix}  e^{2k i\ut} &  e^{(2k-2)i\ut} & e^{(2k+2)i\ut} &  e^{(2k-4)i\ut} & \cdots &  e^{2(2k-1)i\ut} & \underline{1} \end{pmatrix},
\end{align*}
where $\underline{1}$ is the column of all $1$'s.
After re-ordering rows by a permutation $\sigma$ and taking the determinant of the resulting Vandermonde matrix, we have
\[ \det(A)  = \sign(\sigma)\frac{i^k}{2^{k}}\omega \det \begin{pmatrix} \underline{1} & e^{2i\ut} & e^{4i\ut} & \cdots & e^{(2(2k-1))i\ut}\end{pmatrix}
= \sign(\sigma)\frac{i^k}{2^{k}}\omega \prod_{1\leq j <l \leq 2k}\left(e^{2it_l}-e^{2it_j}\right). \]
Finally, note
$\omega=\prod_{1\le j<l\le 2k}e^{-i(t_l+t_j)}$, and multiply each term $\left(e^{2it_l}-e^{2it_j}\right)$ above by the factor $e^{-i(t_l+t_j)}$ extracted from $\omega$ to obtain
\begin{align*}
\det(A)  &  = \sign(\sigma)\frac{i^k}{2^{k}} \prod_{1\leq j<l \leq 2k}\left(e^{i(t_l-t_j)}-e^{-i(t_l-t_j)}\right)\\
&  = \sign(\sigma)\frac{i^k}{2^{k}} \prod_{1\leq j<l \leq 2k}2i \sin(t_l-t_j)= 
\kappa \prod_{1\leq j<l\leq 2k}\sin(t_l-t_j)
\end{align*}
where
$\kappa=\sign(\sigma)\tfrac{i^k}{2^{k}}(2i)^{2k^2-k}=\sign(\sigma) i^{2k^2} 2^{2k(k-1)}=\sign(\sigma) 2^{2k(k-1)}$.
\end{proof}

The following corollary is immediate.

\begin{corollary}\label{cor:det}
For $0\leq i\leq 2k$, let $M_{2k,i}$ denote the $2k\times 2k$ matrix obtained by removing the $i$-th column of $M_{2k}$.
Then
\[\det(M_{2k,i})=\kappa \prod_{\substack{0\leq j<l\leq 2k\\ j,l\neq i}}\sin(t_l-t_j),\]
for some nonzero constant $\kappa$ depending only on $k$.
\end{corollary}

\begin{lemma}\label{lem:nullspace} If no two points $t_0,t_1,\ldots,t_{2k}\in S^1$ are equal or antipodal, then the nullspace of $M_{2k}$ is one-dimensional and is spanned by $\vec{\lambda}=(\lambda_0, \lambda_1,\ldots ,\lambda_{2k})^\intercal $, where 
\[\lambda_i=(-1)^i\prod_{\substack{0\leq j<l\leq 2k\\ j,l\neq i}}\sin(t_l-t_j).
\]
\end{lemma}

\begin{proof}
Because $M_{2k}$ has $2k$ rows and $2k+1$ columns, it has nullity at least one.
Further, by Corollary~\ref{cor:det}, observe that $M_{2k,0}$ is invertible if and only if no two points $t_l,t_j\in S^1$ are equal or antipodal.
Hence, $M_{2k}$ contains $2k$ linearly independent columns and has nullity exactly one.

Next, we prove $\vec{\lambda}$ is contained in the nullspace of $M_{2k}$.
To ease notation, write 
\[M_{2k}\vec{\lambda}=
\begin{pmatrix} C_1 &
S_1&
C_3&
S_3&
\cdots &
C_{2k-1}&
S_{2k-1}
\end{pmatrix}^{\text{T}}.\]
Note $\lambda_i=(-1)^i\frac{1}{\kappa}\det(M_{2k,i})$, and hence for $n=1,3,5,\dots,2k-1$ we have
\[
C_n=\sum_{i=0}^{2k} \cos(n t_i)\lambda_i 
= \frac{1}{\kappa} \sum_{i=0}^{2k}(-1)^i \cos(n t_i)\det(M_{2k,i}).
\]
Therefore, $C_n$ is equal to $\frac{1}{\kappa}$ times the determinant of the matrix
\[\begin{pmatrix}
\cos(n t_0) & \cos(n t_1)   & \hdots & \cos(n t_{2k})  \\
\cos(t_0) & \cos(t_1)   & \hdots & \cos(t_{2k})  \\
\sin(t_0) & \sin(t_1)   & \hdots & \sin(t_{2k})  \\
\cos(3t_0) & \cos(3t_1) & \hdots & \cos(3t_{2k}) \\
\sin(3t_0) & \sin(3t_1) & \hdots & \sin(3t_{2k}) \\
\vdots & \vdots & \ddots & \vdots \\
\cos((2k-1)t_0) & \cos((2k-1)t_1) & \hdots & \cos((2k-1)t_{2k}) \\
\sin((2k-1)t_0) & \sin((2k-1)t_1) & \hdots & \sin((2k-1)t_{2k})
\end{pmatrix}.\]
Since $n=2j-1$ for some $1 \le j \le k$, the first row of this matrix is equal to one of the other rows.
Hence, the matrix is singular, giving that $C_n=0$.

Similarly, it follows that $S_n$ is equal to $\frac{1}{\kappa}$ times the determinant of the same matrix, except with the first row replaced by $(\sin(n t_0), \sin(n t_1), \ldots, \sin(n t_{2k}))$.
For the same reasons as before, it follows that $S_n=0$.
\end{proof}

For convenience, we rescale $\vec{\lambda}$ by $\gamma:=\prod_{0\leq j < l \leq 2k} \frac{1}{\sin(t_l-t_j)}$ (which is well-defined for $t_1,\ldots,t_{2k}$ distinct) to obtain
\[\gamma \vec{\lambda}=\left(\frac{1}{\alpha_0(t_0,\dots,t_{2k})},\dots,\frac{1}{\alpha_{2k}(t_0,\dots,t_{2k})}\right)^\intercal ,
\quad\text{where}\quad
\alpha_i(t_0,\dots,t_{2k})=\prod_{\substack{0\leq j\leq 2k\\ j\neq i}} \sin(t_j-t_i).\]

Recall that entries of $\vec{\lambda}$ correspond to coefficients in the linear combination $\vec{0}=\sum_{i=0}^{2k} \lambda_i \sm_{2k}(t_i)$.
In particular, we are concerned only with \emph{convex} linear combinations.
Hence, after normalizing $\vec{\lambda}$ (and potentially rescaling by $-1$), it is necessary that each entry $\lambda_i$ is positive.
In other words, the origin may be contained in the convex hull of $\{\sm_{2k}(t_0),\dots,\sm_{2k}(t_{2k})\}$ only in the case that the terms $\alpha_i(t_0,\dots,t_{2k})$ share the same sign.
We next relate the sign of each term $\alpha_i(t_0,\dots,t_{2k})$ to the configuration of points $t_0,\dots,t_{2k}\in S^1$.

\begin{lemma}\label{lem:alphasign}
Let $t_0,\dots,t_{2k}\in S^1$, with no two points equal or antipodal.
Then, the numbers $\alpha_i(t_0,\dots,t_{2k})$ have the same sign for all $0\leq i\leq 2k$ if and only if $\chi(t_i):=\#\{t_j\mid t_j\in (t_i+\pi, t_i)_{S^1}\}=k$ for all $i$.
\end{lemma}

\begin{proof}
Throughout, we assume that the points $t_0,\dots,t_{2k}\in S^1$ are distinct, with no two points antipodal, and furthermore that they are ordered by index with a counterclockwise orientation.
Observe that $\sign(\alpha_i(t_0,\dots,t_{2k}))=(-1)^{\chi(t_i)}$.

We first prove two preliminary properties.
\begin{itemize}
\item[(i)] $\sum_{i=0}^{2k}\chi(t_i)=k(2k+1)$.
\item[(ii)] If $t_0,\dots,t_{2k}$ are not all contained in a semicircle, then $1\geq \chi(t_{i+1})-\chi(t_i)$ for $0\leq i \leq 2k$, where we set $t_{2k+1}=t_0$.
\end{itemize}

For (i), note that since no two points are equal or antipodal, we have that $t_j\in (t_i+\pi, t_i)_{S^1}$ if and only if $t_i\notin (t_j+\pi, t_j)_{S^1}$.
Therefore
$\sum_{i=0}^{2k}\chi(t_i)=\binom{2k+1}{2}=k(2k+1)$.

For (ii), observe that the open arc $(t_{i+1}+\pi,t_i)_{S^1}$ contains exactly $\chi(t_{i+1})-1$ points.
Indeed, $(t_{i+1}+\pi,t_i)_{S^1}$ contains exactly $\chi(t_{i+1})-1$ points for all $i$ if and only if $t_i\in (t_{i+1}+\pi,t_{i+1})_{S^1}$ for all $i$, which is true if and only if the points are not contained in a semicircle.
Hence, $(t_i+\pi,t_{i+1}+\pi)_{S^1}$ must contain exactly $\chi(t_i)-(\chi(t_{i+1})-1)$ points.
Because this number is non-negative, it follows that $1\geq \chi(t_{i+1})-\chi(t_i)$.

We now prove Lemma~\ref{lem:alphasign}.
In the case that $\chi(t_i)=k$ for all $i$, we see that the numbers $\alpha_i(t_0,\dots,t_{2k})$ are all positive or are all negative.

Conversely, suppose the numbers $\alpha_i(t_0,\dots,t_{2k})$ have the same sign.
Since $\sign(\alpha_i(t_0,\dots,t_{2k}))=(-1)^{\chi(t_i)}$, the numbers $\chi(t_i)$ have the same parity.
Further, in the case $k$ is odd (resp.\ even), (i) implies each $\chi(t_i)$ is odd (resp.\ even).
Therefore, in either case, we may write $\chi(t_i)=k+2n_i$ for some integer $n_i\in\Z$.
Note that (i) implies 
\[k(2k+1)=\sum_{i=0}^{2k}\chi(t_i)=\sum_{i=0}^{2k}(k+2n_i)=k(2k+1)+2\sum_{i=0}^{2k}n_i, \]
giving $\sum_{i=0}^{2k}n_i=0$.
Therefore, it is sufficient to prove that $n_i=n_j$ for all $i,j$.
Toward that end, define $t_{2k+1}=t_0$ and $n_{2k+1}=n_0$, and observe
\begin{align*}
0=\sum_{i=0}^{2k}n_{i+1} =\sum_{i=0}^{2k}(n_{i+1}+(-n_{i}+n_{i}))
=\sum_{i=0}^{2k}((n_{i+1}-n_{i})+n_{i})
=\sum_{i=0}^{2k}(n_{i+1}-n_{i})+\sum_{i=0}^{2k}n_{i}
=\sum_{i=0}^{2k}(n_{i+1}-n_{i}).
\end{align*}
It cannot be the case that all of the points $t_i$ are contained in a semicircle, since then $\chi(t_i)$ would obtain all of the values $0,1,\dots, 2k$, contradicting the fact that these values have the same parity.
Therefore, we may apply (ii) to obtain 
\[1\geq (k+2n_{i+1})-(k+2n_i)=2(n_{i+1}-n_i),\]
which implies $0\geq n_{i+1}-n_i$ for all $i$.
Together with $\sum_{i=0}^{2k}(n_{i+1}-n_i)=0$, this gives $n_{i+1}=n_i$ for all $i$.
\end{proof}

We are now prepared to prove Theorem~\ref{thm:miss-origin}.

\begin{proof}[Proof of Theorem~\ref{thm:miss-origin}] 
Let distinct $t_0,\dots, t_{2k}\in S^1$ be given in counterclockwise order, and define $D=\diam(\{t_0,\dots, t_{2k}\})$.
We claim that if $\chi(t_i):=\#\{t_j\mid t_j\in (t_i+\pi, t_i)_{S^1}\}=k$ for all $i$, then $D\ge \frac{2\pi k}{2k+1}$. 
Indeed, define $t_{2k+1}=t_0,$ and let $\ell_i$ be the length of $(t_{i},t_{i+1})_{S^1}$ for all $i$. 
Because $\chi(t_i)=k=\chi(t_{i+1})$, it follows that there exists exactly one point $t_j$ in the arc $(t_i+\pi,t_{i+1}+\pi)_{S^1}.$
Further, because the function $f\colon(t_i+\pi,t_{i+1}+\pi)_{S^1}\to \R$ defined by $f(t)=\max\{d_{S^1}(t,t_i),d_{S^1}(t,t_{i+1})\}$ is minimized at the midpoint of $(t_i+\pi,t_{i+1}+\pi)_{S^1}$, it follows that $D\ge\pi-\frac{\ell_i}{2}$.
On the other hand, because there are $2k+1$ consecutive pairs of points $t_i,t_{i+1}$, we must have $\ell_j\leq \frac{2\pi}{2k+1}$ for some $0\leq j \leq 2k$. 
Hence $D\ge\pi-\frac{\pi}{2k+1}=\frac{2\pi k}{2k+1}$.

Therefore, if $\diam(\{t_0,\dots, t_{2k}\})<\mathcal{C}= \frac{2\pi k}{2k+1}$, then $\chi(t_i)\neq k$ for some $0\leq i \leq 2k$.
Hence Lemmas~\ref{lem:nullspace} and~\ref{lem:alphasign} imply that there do not exist positive scalars $\lambda_i$ with $\vec{0}=\sum_{i=0}^{2k} \lambda_i \sm_{2k}(t_i)$.

To see that this bound is sharp, let $t_i\in S^1$ denote the vertices of a regular inscribed $(2k+1)$-gon.
Note that $\vec{0}=\sum_{i=0}^{2k}\frac{1}{2k+1}\sm_{2k}(t_i)$ in this case.
\end{proof}

\section{A connection between metric thickenings and orbitopes}
\label{sec:results}

We now connect the Vietoris--Rips metric thickenings of the circle to the Barvinok--Novik orbitopes.
Indeed, we conjecture (Conjecture~\ref{conj:homotopy-S1}) that for $\frac{2\pi(k-1)}{2k-1}\le r<\frac{2\pi k}{2k+1}$, the metric thickening $\vrm{S^1}{r}$ is homotopy equivalent to the boundary $\partial \cB_{2k}$ of the Barvinok--Novik orbitope, i.e.\ to the odd-dimensional sphere $S^{2k-1}$.
We are able to show the partial result that the $(2k-1)$-dimensional homology, cohomology, and homotopy groups of $\vrm{S^1}{r}$ are nontrivial; we only obtain the full homotopy type for $r\le\frac{2\pi}{3}$.

\begin{figure}[h]
\centering
\def\svgwidth{0.9\linewidth}
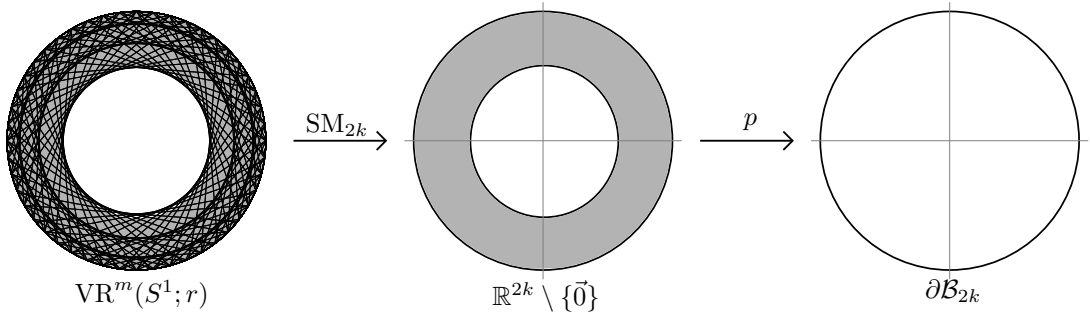
\caption{The composition of maps $\vrm{S^1}{r}\xrightarrow{\sm_{2k}}\R^{2k}\setminus\{\vec{0}\}\xrightarrow{\,p\,}\partial \cB_{2k}$, drawn in the case $k=1$.}
\label{fig:BarvinokNovikProjection}
\end{figure}

Towards Conjecture~\ref{conj:homotopy-S1}, we build the following sequence of maps: 
\[\vrm{S^1}{r}\xrightarrow{\sm_{2k}}\R^{2k}\setminus\{\vec{0}\}\xrightarrow{\,p\,}\partial \cB_{2k}\xrightarrow{\,\iota\,} \vrm{S^1}{r}.\]
This construction will proceed as outlined below.

\begin{enumerate}
\item Section~\ref{ssec:composition}: We define the radial projection map $p\colon\R^{2k}\setminus \{\vec{0}\}\to \partial\cB_{2k}$.
We extend the domain of $\sm_{2k}$ to $\vrm{S^1}{r}$, and note that the composition $p\circ\sm_{2k}$ is well-defined.
\item Section~\ref{ssec:inclusion}: We define the inclusion $\iota\colon \partial\cB_{2k} \to \vrm{S^1}{r}$.
Since $(p\circ \sm_{2k})\circ\iota=\mathrm{id}_{\partial\cB_{2k}}$, we obtain that the $(2k-1)$-dimensional homology, cohomology, and homotopy groups of $\vrm{S^1}{r}$ are nontrivial.
\item Section~\ref{ssec:inverses}: We prove that $p\circ \sm_{2k}$ and $\iota$ are homotopy inverses; this is the step that we can currently only complete for $r\le\frac{2\pi}{3}$ (and hence $k\le 2$).
\end{enumerate}

When $r<\frac{2\pi}{3}$ and $k=1$, this proof is quite easy to interpret.
The map $\sm_2$ maps the space $\vrm{S^1}{r}$ to an annulus missing the origin in $\R^2$ (see Figure~\ref{fig:BarvinokNovikProjection}).
Map $p$ radially projects the annulus to its outer circle, and map $\iota$ includes the circle back into $\vrm{S^1}{r}$.

As a result of step (3), we obtain $\vrm{S^1}{\frac{2\pi}{3}}\simeq S^3$.
Note that $\vrm{S^1}{\frac{2\pi}{3}}\not\simeq\vr{S^1}{\frac{2\pi}{3}}\simeq\bigvee^{\mathfrak{c}}S^2$.
We think of the metric thickening $\vrm{S^1}{\frac{2\pi}{3}}$ as having the ``right" homotopy type, whereas the wild homotopy type of the simplicial complex $\vr{S^1}{\frac{2\pi}{3}}$ is an artifact of it being equipped with the ``wrong" topology\footnote{The inclusion $S^1\hookrightarrow \vr{S^1}{\frac{2\pi}{3}}$ is not continuous.
As further evidence for being more interested in the homotopy type $S^3$ rather than $\bigvee^\mathfrak{c} S^2$, note that for all $0<\varepsilon<\frac{2\pi}{15}$ we have $\vr{S^1}{\frac{2\pi}{3}+\varepsilon}\simeq S^3$.}.

There is an analogous relationship between the \v{C}ech thickenings of the circle and the Carath\'{e}odory orbitopes, i.e.\ the convex hull of the curve $(\cos t, \sin t, \cos 2t, \sin 2t, \ldots, \cos kt, \sin kt\bigr)$~\cite{sanyal2011orbitopes}.
We do not detail that connection here, although some connections between \v{C}ech complexes of finite points on the circle and cyclic polytopes (convex hulls of finite sets of points from this trigonometric moment curve) are given in~\cite{AAFPP-J}.

\subsection{Map from the Vietoris--Rips thickening to the Barvinok--Novik orbitope}\label{ssec:composition}

We first define the radial projection map $p\colon\R^{2k}\setminus \{\vec{0}\}\to \partial\cB_{2k} \simeq S^{2k-1} $.
As $\cB_{2k}$ is a convex body containing the origin in its interior, each ray emanating from the origin intersects $\partial\cB_{2k}$ exactly once.
Hence, $p$ is well-defined.

We extend $\sm_{2k}\colon S^1\to\R^{2k}$ to $\sm_{2k}\colon\vrm{S^1}{r}\to\R^{2k}$ by declaring
$\sm_{2k}\left(\sum_i\lambda_i \delta_{t_i}\right) = \sum_i \lambda_i \sm_{2k}(t_i)$.
Here the sum on the left-hand side defines a measure as a convex sum of Dirac delta functions at the points $t_i\in S^1$ (of diameter at most $r$), whereas the sum on the right-hand side is a sum of vectors in $\R^{2k}$.
Because $\sm_{2k}$ restricted to $S^1$ is continuous and bounded, Lemma 5.2 of~\cite{AAF} proves that this extension to all of $\vrm{S^1}{r}$ is continuous.

Finally, suppose $r < \frac{2\pi k}{2k+1}$.
Then, Theorem~\ref{thm:miss-origin} implies that the origin $\vec{0}\in \R^{2k}$ is not in the image of the map $\sm_{2k}\colon\vrm{S^1}{r}\to\R^{2k}$, and hence the composition $p\circ \sm_{2k}\colon \vrm{S^1}{r}\to \partial\cB_{2k}$ from the Vietoris--Rips thickening to the boundary of the Barvinok--Novik orbitope is well-defined.

\subsection{Inclusion from the Barvinok--Novik orbitope boundary to the Vietoris--Rips thickening}\label{ssec:inclusion}

For $r\ge\frac{2\pi(k-1)}{2k-1}$, we define the map $\iota\colon \partial\cB_{2k} \to \vrm{S^1}{r}$ as follows.
Given a point $\sum_i\lambda_i\sm_{2k}(t_i)\in \partial\cB_{2k}$ with $\lambda_i>0$ for all $i$, let $\iota\left(\sum_i\lambda_i\sm_{2k}(t_i)\right)=\sum_i\lambda_i \delta_{t_i}$.
Recall that Corollary~\ref{cor:inclusion} states every face of $\cB_{2k}$ is a simplex whose diameter in $S^1$ (not in $\R^{2k}$) is at most $\frac{2\pi (k-1)}{2k-1}$, and hence the image of $\iota$ indeed lands in $\vrm{S^1}{r}$.
We now prove that $\iota$ is continuous.

\begin{lemma}\label{lem:iota-cont}
Let $r\ge\frac{2\pi(k-1)}{2k-1}$.
The map $\iota\colon \partial\cB_{2k} \to \vrm{S^1}{r}$ is continuous.
\end{lemma}

\begin{proof}
Recall $p$ is the radial projection to the boundary of $\cB_{2k}$.
We will show that
\[(p\circ \sm_{2k})|_{\iota(\partial\cB_{2k})}\colon\iota(\partial\cB_{2k})\to\partial\cB_{2k}\]
is a bijective continuous function from a compact space to a Hausdorff space.
It will then follow from~\cite[Theorem~3.7]{armstrong2013basic} that $(p\circ \sm_{2k})|_{\iota(\partial\cB_{2k})}$ is a homeomorphism, with a continuous inverse $\iota \colon\partial\cB_{2k}\to\iota(\partial\cB_{2k})$.
Therefore $\iota\colon \partial\cB_{2k} \to \vrm{S^1}{r}$ is continuous.

The fact that $(p\circ \sm_{2k})|_{\iota(\partial\cB_{2k})}$ is a bijective function follows from Corollary~\ref{cor:inclusion}.
The space $\partial\cB_{2k}$ is Hausdorff since it inherits the subspace topology from Euclidean space.
Finally, to see that $\iota(\partial\cB_{2k})$ is compact, we note that $\iota(\partial\cB_{2k})$ is a closed subset of $\sP(S^1)$, the space of all Radon probability measures on $S^1$ equipped with the Wasserstein metric.
Since $S^1$ is compact, it follows that $\sP(S^1)$ is compact by~\cite[Remark~6.19]{villani2008optimal}, and therefore $\iota(\partial\cB_{2k})$ is compact as a closed subset of a compact space.
\end{proof}

We can now give the following corollary of Theorem~\ref{thm:miss-origin}.

\begin{corollary}\label{cor:nontrival-homology}
For $\frac{2\pi(k-1)}{2k-1}\le r < \frac{2\pi k}{2k+1}$, the $(2k-1)$-dimensional homology, cohomology, and homotopy groups of $\vrm{S^1}{r}$ are nontrivial.
\end{corollary}

\begin{proof}
Theorem~\ref{thm:miss-origin} implies that for in this range of $r$ values, the map $(p\circ\sm_{2k})\circ\iota$ is the identity map on $\partial\cB_{2k}$, i.e.\ that the space $\partial \cB_{2k}\cong S^{2k-1}$ is a retract of $\vrm{S^1}{r}$.
\end{proof}

\subsection{Show $p\circ \sm_{2k}$ and $\iota$ are homotopy inverses.}\label{ssec:inverses}

We conjecture that the composition $\iota\circ p\circ\sm_{2k}$ has a controllable effect on the diameter of any measure in the Vietoris--Rips thickening.

\begin{conjecture}\label{conj:diam}
Given $\frac{2\pi (k-1)}{2k-1}\leq r < \frac{2\pi k}{2k+1}$ and $\mu\in\vrm{S^1}{r}$, we conjecture \[\diam(\supp(\mu))=\diam(\supp(\mu)\cup\supp(\iota \circ p\circ \sm_{2k}(\mu))).\]
\end{conjecture}

\begin{theorem}
Conjecture~\ref{conj:diam} would imply Conjecture~\ref{conj:homotopy-S1}, namely that for $\frac{2\pi(k-1)}{2k-1}\le r<\frac{2\pi k}{2k+1}$, we have
\[\vrm{S^1}{r}\simeq\partial \cB_{2k}\cong S^{2k-1}.\]
\end{theorem}

\begin{proof}
As observed in the proof of Corollary~\ref{cor:nontrival-homology}, we have that $(p\circ \sm_{2k})\circ\iota=\mathrm{id}_{\partial\cB_{2k}}$.
Hence, it remains only to show that $\iota \circ (p\circ \sm_{2k})\simeq \mathrm{id}_{\vrm{S^1}{r}}$.
We will do so using the simplest possible homotopy, a linear homotopy.
Indeed, consider the linear homotopy $H\colon \vrm{S^1}{r}\times I\to \vrm{S^1}{r}$ defined by \[H(\mu,t)=(1-t)\mu+t[\iota \circ (p\circ \sm_{2k})(\mu)].\]
Conjecture~\ref{conj:diam} would imply that $H$ is well-defined, and hence also continuous by Lemma~3.8 of~\cite{AAF}.
Note $H(-,0)=\mathrm{id}_{\vrm{S^1}{r}}$ and $H(-,1)=\iota \circ (p\circ \sm_{2k})$.
Hence, this would imply $\vrm{S^1}{r}\simeq \partial B_{2k}\cong S^{2k-1}$.
\end{proof}

We remark that Conjecture~\ref{conj:diam} is true for $r<\frac{2\pi}{3}$, giving $\vrm{S^1}{r}\simeq S^1$ for $r<\frac{2\pi}{3}$.

\medskip

The remainder of this section is devoted to proving that Conjecture~\ref{conj:diam} is true for $r=\frac{2\pi}{3}$, and hence $\vrm{S^1}{\frac{2\pi}{3}}\simeq S^3$.
In order to prove this, we first describe a number of intermediate lemmas.
 
The first such lemma, Farkas' Lemma, characterizes when a vector lies in the convex cone generated by a set of vectors.
Let $\R^+=\{t\in\R~|~t\ge0\}$.

\begin{lemma}[Farkas' Lemma~\cite{boyd2004convex}]\label{lem:farkas}
Let $A\in \R^{m\times n}$, let $a_i\in\R^m$ for $1\leq i \leq n$ denote the columns of $A$, and let $v\in \R^m$.
Then, exactly one of the following is true: 
\begin{enumerate}
\item There exists $x\in (\R^+)^n$ such that $Ax=v$.
\item There exists $y\in \R^m$ such that $a_i^\intercal y\geq 0$ for all $i$ and $v^\intercal  y<0$.
\end{enumerate}
\end{lemma}
Case (1) above is equivalent to $v\in\cone(\{a_1,\dots,a_n\})$, and case (2) is equivalent to $v\notin\cone(\{a_1,\dots,a_n\})$.
We can use Farkas' Lemma to study how cones intersect.

\begin{lemma}\label{lem:coneintersect}
Let $u_0,\dots, u_n,v_0,\dots,v_k\in\R^{m}$.
If there exists some $y\in \R^m$ such that $u_i^\intercal y\geq 0$ for $0\leq i\leq n$ and $v_i^\intercal y<0$ for $0\leq i\leq k$,
then $\cone\left(\left\{u_0,\dots, u_n\right\}\right)\cap \cone\left(\left\{v_0,\dots,v_k\right\}\right)=\vec{0}$.
\end{lemma}

\begin{proof}
Suppose such a vector $y\in\R^m$ exists, and let $\vec{0}\neq v=\sum_{i=0}^k \lambda_i v_i\in \cone(\{v_0,\dots,v_k\})$.
Then, because there exists some $0\leq j \leq k$ with $\lambda_j>0$, we have $v^\intercal  y=\sum_{i=0}^k \lambda_i v_i^\intercal y\leq \lambda_j  v_j^\intercal y <0$.
Hence, by Lemma~\ref{lem:farkas}, $v$ is not contained in the convex cone generated by $\{u_0,\dots,u_n\}$.
\end{proof}

The following theorem will be used to construct a vector satisfying the hypotheses of Lemma~\ref{lem:coneintersect}, given certain configurations of points along the curve $\sm_{2k}$.

\begin{theorem}\label{thm:farkasvector}
Fix a positive integer $k$ and distinct $v_1,\dots, v_{2k-1}\in S^1$ with no two points antipodal.
Let $u_1,\dots, u_{4k-2}$ denote the set of points $\{v_1, \dots, v_{2k-1}\}\cup \{v_1+\pi, \dots, v_{2k-1}+\pi\}$ labeled in counterclockwise order such that $u_1=v_1$.
Then, there exists a raked homogeneous trigonometric polynomial $f$ of degree $2k-1$ such that $f(u_i)=0$ for $1\leq i\leq 4k-2$. 
Further, $\sign(f(t))=(-1)^i$ for $t\in (u_i,u_{i+1})_{S^1}$, where we define $u_{4k-1}=u_1$.
\end{theorem}

\begin{proof}
For $t\in S^1$, consider points $\sm_{2k}(t)\in\R^{2k}$ to be written as column vectors and define the $2k\times 2k$ matrix
\[N(t)=\begin{pmatrix}
\sm_{2k}(t) & \sm_{2k}(v_1) & \sm_{2k}(v_2) & \cdots & \sm_{2k}(v_{2k-2}) & \sm_{2k}(v_{2k-1}) 
\end{pmatrix}.
\]
By Lemma~\ref{lem:det}, 
\[\det(N(t))=\kappa \Biggl(\prod_{\substack{1\leq j<l\leq 2k-1 }}\sin(v_l-v_j) \Biggr)\Biggl(\prod_{\substack{1\leq l\leq 2k-1 }} \sin(v_l - t) \Biggr),
\]
where $\kappa$ is a nonzero constant that depends only on $k$.
Further, by considering the cofactor expansion of this determinant along the first column of $N(t)$, observe that $\det(N(t))$ is a raked homogeneous trigonometric polynomial of degree $2k-1$.
Because no two elements of  $\{v_1,\ldots,v_{2k-1}\}$ are equal or antipodal, note that $\prod_{\substack{1\leq j<l\leq 2k-1 }}\sin(v_l-v_j)\neq 0$.
Hence,
\[f(t)=\frac{1}{\kappa} \Biggl(\prod_{\substack{1\leq j<l\leq 2k-1 }}\sin(v_l-v_j) \Biggr)^{-1}\det(N(t))=\prod_{\substack{1\leq l\leq 2k-1 }} \sin(v_l - t)\]
is a well-defined raked homogeneous trigonometric polynomial of degree $2k-1$ with real roots $\{v_1,\dots, v_{2k-1}\}\cup \{v_1+\pi,\dots, v_{2k-1}+\pi\}$.
Finally, observe for $t\notin \{v_1, \dots, v_{2k-1}\}\cup \{v_1+\pi, \dots, v_{2k-1}+\pi\}$ we have
\[\sign\left(f(t)\right)=\sign\Biggl(\prod_{\substack{1\leq l\leq 2k-1 }} \sin(v_l - t) \Biggr)=(-1)^{\rho(t)},
\]
where we define $\rho(t)=\#\{v_l\mid v_l\in(t+\pi,t)_{S^1}, \,1\le l\le 2k-1 \}$.
\end{proof}

\begin{remark}\label{rmk:farkasvector}
In the setting of Theorem~\ref{thm:farkasvector}, there exists a vector $y\in\R^{2k}$ such that $\left(\sm_{2k}(u_i)\right)^\intercal y=0$ for all $i$.
Further, $\sign\left(\left(\sm_{2k}(t)\right)^\intercal y\right)=(-1)^i$ for $t\in (u_i,u_{i+1})_{S^1}$, where we define $u_{4k-1}=u_1$.
\end{remark}

\begin{proposition}\label{prop:cones-one-third}
Let distinct $t_1,\dots, t_{n}\in S^1$ be in counterclockwise order and contained in an arc $[t_1,t_n]_{S^1}$ of length at most $\frac{2\pi}{3}$.
Let distinct $s_1,\dots,s_{m}\in S^1$ be such that $\conv(\sm_4(\{s_1,\dots,s_m\}))$ is a face of $\cB_4$, and $\{s_1,\dots,s_{m}\}\nsubseteq[t_1,t_n]_{S^1}$.
Then
\[\cone\left(\sm_4(\{s_1,\dots,s_m\})\right)\cap \cone\left(\sm_4(\{t_1,\dots,t_n\})\right)=\cone\left(\sm_4(\{s_1,\dots,s_m\}\cap \{t_1,\dots,t_n\})\right).\]
\end{proposition}

For the above proposition we agree $\cone(\varnothing)=\vec{0}$. 

\begin{proof}
Throughout, for convenience, consider points $\sm_{4}(t)\in\R^{4}$ to be written as column vectors. 
In light of the known facial structure of $\cB_4$ (Theorem~\ref{thm:faces}), it follows that $m\le 3$. 
Hence, there are three cases:
\begin{enumerate}[(i)]
\item The sets $\{s_1,\dots, s_m\}$ and $\{t_1,\dots, t_n\}$ are disjoint.
\item The sets $\{s_1,\dots, s_m\}$ and $\{t_1,\dots, t_n\}$ contain one point of intersection.
In this case, $m\in\{2,3\}$, i.e.\ $\{s_1,\dots, s_m\}$ determines an edge or an equilateral triangle in $\partial\cB_4$.
\item The sets $\{s_1,\dots, s_m\}$ and $\{t_1,\dots, t_n\}$ contain two points of intersection.
In this case, $m=3$, the points $\{s_1, s_2, s_3\}$ determine an equilateral triangle in $\partial\cB_4$, $\{s_1, s_2, s_3\}\cap \{t_1,\dots,t_n\}=\{t_1,t_n\}$, and the length of $(t_1,t_n)_{S^1}$ is $\frac{2\pi}{3}$.
\end{enumerate}

The proof will proceed as follows. 
We will consider first the case that $\{t_1,\dots, t_n\}$ and $\{s_1,\dots, s_m\}$ are disjoint and apply Lemma~\ref{lem:coneintersect} to prove that the resulting cones in $\R^4$ must be disjoint.
Then, we will generalize this argument to allow for intersections and consider the remaining two cases.

Toward that end, suppose $\{s_1,\dots, s_m\}\cap\{t_1,\dots, t_n\}=\varnothing$ and note, by Lemma~\ref{lem:coneintersect}, that it is sufficient to find $y\in\R^{4}$ such that $\left(\sm_{4}(t_i)\right)
^\intercal y\geq 0$ for $1\leq i\leq n$ and $\left(\sm_{4}(s_i)\right)
^\intercal y< 0$ for $1\leq i \leq m$.
To define such a vector $y$, fix points $v_1,v_2,v_3\in S^1$ as follows.
By the assumptions on the configuration of the points $\{s_1,\dots,s_{m}\}$, observe there must exist an arc $\Gamma=(\gamma_1,\gamma_2)_{S^1}$ of length $\pi$ such that 
\begin{itemize}
\item $[t_1,t_n]_{S^1}\subseteq \Gamma$, 
\item $\{s_1,\dots, s_m\}\cap \{\gamma_1,\gamma_2\}=\varnothing$, and 
\item $|\{s_1,\dots, s_m\}\cap \Gamma|=N$ for $N\leq 1$.
\end{itemize}
Indeed, to see that we can arrange $N\le 2$, note that if $m=3$ then $\{s_1, s_2, s_3\}$ are the vertices of an equilateral triangle, and hence not in an arc of length $\pi$.
To see that we can arrange $N\le 1$, note that if $m=2$, then since one of the $s_i$ points is outside $[t_1,t_n]_{S^1}$, we can choose $\Gamma$ so that the same $s_i$ point is also outside $\Gamma$.

\begin{figure}[h!]
\begin{minipage}{.5\textwidth}
\begin{center}
\begin{overpic}[width=0.8\textwidth]{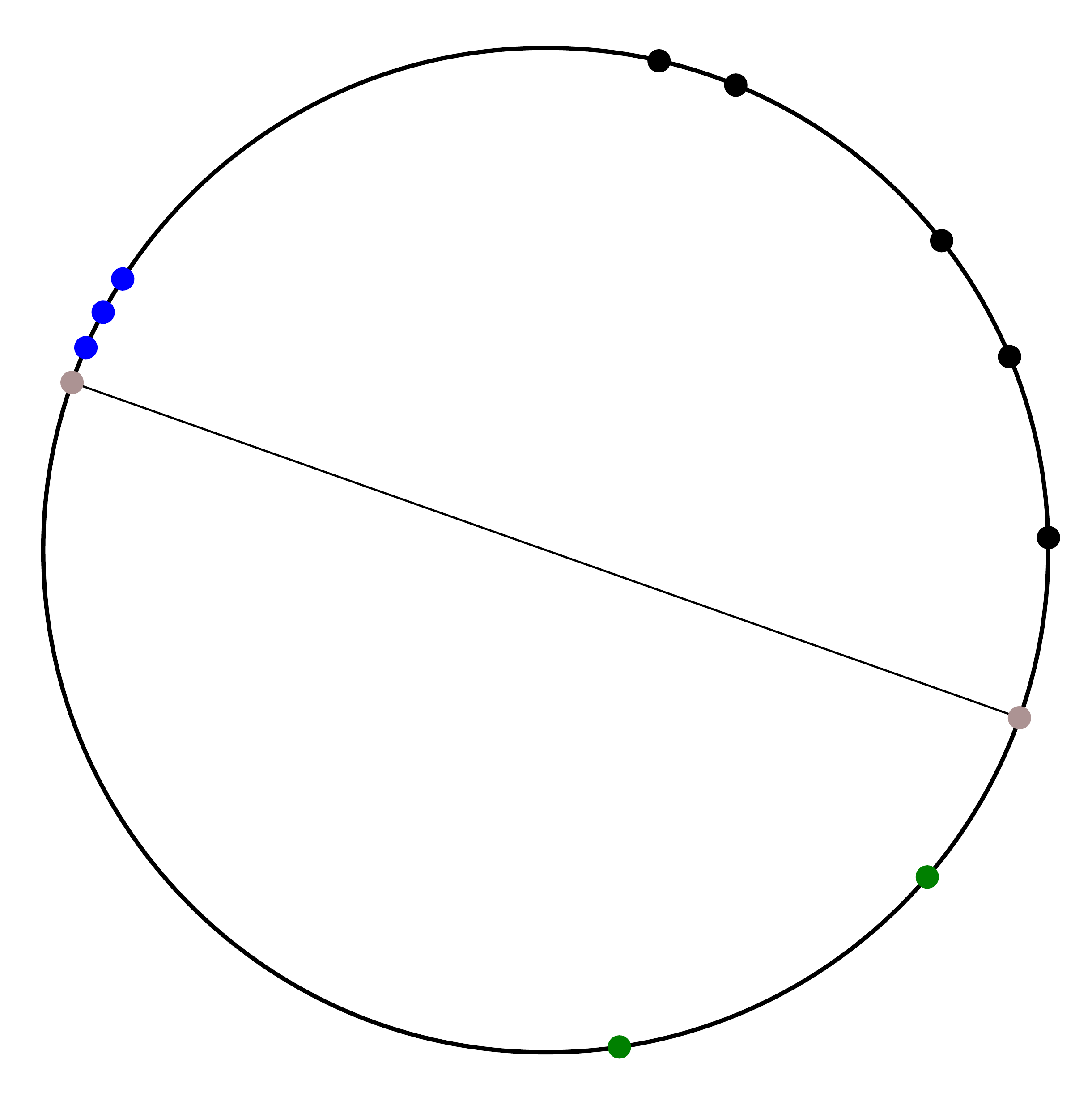}
\put (97.7,50) {\footnotesize $\displaystyle t_1$}
\put (94,67.5) {\footnotesize $\displaystyle t_2$}
\put (88,78) {\footnotesize $\displaystyle t_3$}
\put (67,94.3) {\footnotesize $\displaystyle t_4$}
\put (60,96.5) {\footnotesize $\displaystyle t_5$}
\put (13.3,72.8) {\footnotesize $\displaystyle v_1$}
\put (11,69.5) {\footnotesize $\displaystyle v_2$}
\put (10,66) {\footnotesize $\displaystyle v_3$}
\put (56,0) {\footnotesize $\displaystyle s_1$}
\put (86,17) {\footnotesize $\displaystyle s_2$}
\put (95,33.3) {\footnotesize $\displaystyle \gamma_1$}
\put (0.5,66) {\footnotesize $\displaystyle \gamma_2$}
\end{overpic}
\end{center}
\caption*{An example of points $\{t_1,\dots, t_{5}\}$ and $\{s_1, s_2\}$ in $S^1$ in the case $N=0$.
}
    \end{minipage}%
    \begin{minipage}{0.5\textwidth}
\begin{center}
\begin{overpic}[width=0.8\textwidth]{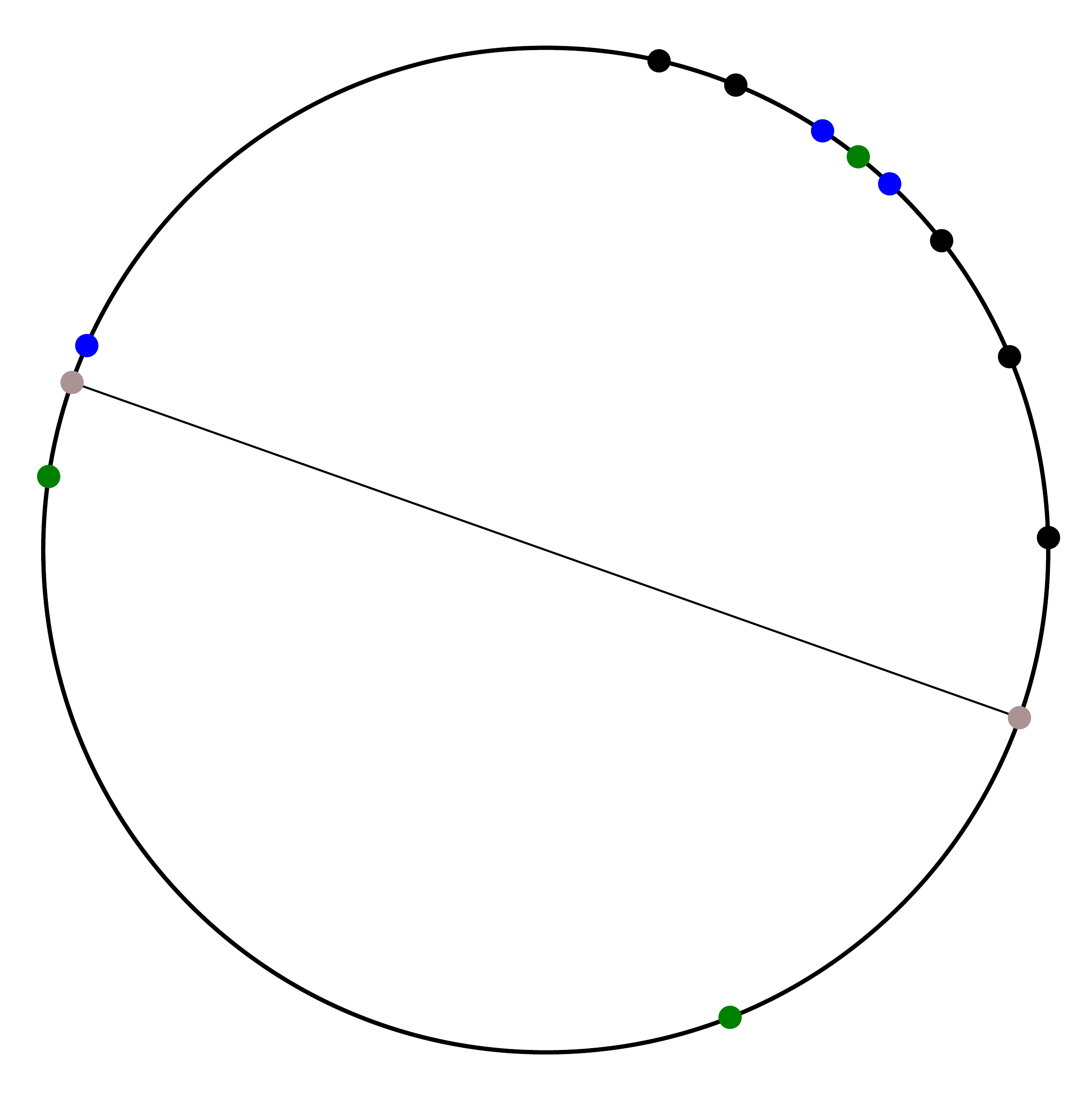}
\put (97.7,50) {\footnotesize $\displaystyle t_1$}
\put (94,67.5) {\footnotesize $\displaystyle t_2$}
\put (88,78) {\footnotesize $\displaystyle t_3$}
\put (78,79) {\footnotesize $\displaystyle v_1$}
\put (78.6,87.4) {\footnotesize $\displaystyle s_1$}
\put (71.5,84.2) {\footnotesize $\displaystyle v_2$}
\put (67,94.3) {\footnotesize $\displaystyle t_4$}
\put (60,96.5) {\footnotesize $\displaystyle t_5$}
\put (10.2,67) {\footnotesize $\displaystyle v_3$}
\put (-2,56) {\footnotesize $\displaystyle s_2$}
\put (67,3.7) {\footnotesize $\displaystyle s_3$}
\put (95,33.3) {\footnotesize $\displaystyle \gamma_1$}
\put (0.5,66) {\footnotesize $\displaystyle \gamma_2$}
\end{overpic}
\end{center}
\caption*{An example of points $\{t_1,\dots, t_{5}\}$ and $\{s_1, s_2, s_3\}$ in $S^1$ in the case $N=1$.
}
    \end{minipage}
\caption{In both cases, the points $\{\gamma_1,\gamma_2\}$ and $\{v_1,\dots, v_3\}$ are defined in the proof of Proposition~\ref{prop:cones-one-third} and are used to construct a vector satisfying the hypotheses of Lemma~\ref{lem:coneintersect}.}
    \end{figure}

If $N=0$, define $v_1=\gamma_2-\delta$, with $\delta>0$ small enough such that both $(v_1+\pi,v_1)_{S^1}\cap\{s_1,\dots, s_m\}=\varnothing$ and $(v_1+\pi,v_1)_{S^1}\cap\{t_1,\dots, t_n\}=\{t_1,\dots, t_n\}$.
Then, define $v_2$ and $v_3$ so that $v_1$, $v_2$, $v_3$, and $\gamma_2$ appear in counterclockwise order.

If $N=1$, assume without loss of generality that $\Gamma\cap \{s_1,\dots, s_m\}=\{s_1\}$.
Then, define $v_{1}=s_1-\varepsilon$ and $v_{2}=s_1+\varepsilon$.
Choose $\varepsilon>0$ small enough such that  $(v_{1},v_{2})_{S^1}$ does not contain any point in $\{t_1,t_2,\dots, t_{n},\gamma_1,\gamma_2\}$ and furthermore so that $(v_{1}+\pi, v_{2}+\pi)_{S^1}\cap \{s_1,\dots, s_{m}\}=\varnothing$.
Such points must exist because no two elements of $\{s_1,\dots,s_{m}\}$ are antipodal.
Finally, define $v_3=\gamma_2-\delta$, with $\delta>0$ small enough such that both $(v_3+\pi,v_3)_{S^1}\cap\{s_1,\dots, s_m\}=\{s_1\}$ and $(v_3+\pi,v_3)_{S^1}\cap\{t_1,\dots, t_n\}=\{t_1,\dots, t_n\}$.

Now, apply Remark~\ref{rmk:farkasvector} to obtain $y\in\R^{4}$ such that for $t\notin\{v_1, v_2, v_3\}\cup \{v_1+\pi, v_2+\pi, v_3+\pi\}$, we have
\[
\sign\left(\left(\sm_{4}(t)\right)^\intercal y\right)=\sign\Biggl(\prod_{\substack{1\leq l\leq 3}} \sin(v_l - t) \Biggr)=(-1)^{\rho(t)}, \]
where $\rho(t)=\#\{v_l\mid v_l\in(t+\pi,t)_{S^1}, \,1\le l\le 3 \}$.
When we consider the case $t=t_i$ for $1\le i\le n$, we note by construction that $\rho(t_i)$ is even for each $t_i$, and so $\left(\sm_{4}(t_i)\right)
^\intercal y\geq 0$ for $1\leq i\leq n$.

On the other hand, in the case $N=0$, we note that $\rho(s_i)=3$ and $\sign\left((\sm_{2k}(s_i))^\intercal  y\right)=-1$ for $1\leq i \leq m$. 
Finally, in the case $N=1$, note that $\rho(s_1)=1$ and $\sign\left((\sm_{2k}(s_1))^\intercal  y\right)=-1$. Further, the pair $\{v_{1}, v_{2}\}$ has zero net effect on the parity of $\rho(s_i)$ for $2\leq i \leq m$ by the fact that $(v_{1}+\pi, v_{2}+\pi)_{S^1}\cap \{s_1,\dots, s_{m}\}=\varnothing$.
Hence, $\sign\left((\sm_{2k}(s_i))^\intercal y\right)=-1$ for $2\leq i \leq m$.

This concludes the proof of case (i) that 
$\cone\left(\sm_4(\{s_1,\dots,s_m\})\right)\cap \cone\left(\sm_4(\{t_1,\dots,t_n\})\right)=\{\vec{0}\}$
when $\{s_1,\dots, s_m\}\cap \{t_1,\dots, t_n\}=\varnothing.$

Next, consider case (ii).
Assume without loss of generality that $\{s_1,\dots, s_m\}\cap\{t_1,\dots, t_n\}=\{s_1\},$
and write $s_1=t_\alpha$ for some $1\leq \alpha \leq n$.
Given
$\vec{u}\in \cone\left(\sm_{4}(\{t_1,\dots,t_{n}\})\right)\cap \cone\left(\sm_{4}(\{s_1,\dots,s_{m}\})\right)$,
write $\vec{u}=\sum_{i=1}^n \lambda_i \sm_{4}(t_i)=\sum_{j=1}^m \kappa_j \sm_{4}(s_j)$ for some non-negative scalars $\lambda_i,\kappa_j$. 
To show $\vec{u}\in \cone\left(\sm_4(t_\alpha)\right)$, observe that it is sufficient to prove $\lambda_i=0$ for all $i\in\{1,\dots, n\}\setminus\alpha$.
We consider the possibilities $\lambda_\alpha\geq \kappa_1$ and $\lambda_\alpha< \kappa_1$ separately.

If $\lambda_\alpha\geq \kappa_1$, then
\[\vec{u}-\kappa_1\sm_4(s_1)=(\lambda_\alpha-\kappa_1)\sm_4(t_\alpha) +\sum_{i\in\{1,\dots, n\}\setminus \alpha}\lambda_i \sm_{4}(t_i)=\sum_{j=2}^m \kappa_j \sm_{4}(s_j).\]
It follows that $\vec{u}-\kappa_1\sm_4(s_1)\in \cone\left(\sm_{4}(\{t_1,\dots,t_{n}\})\right)\cap \cone\left(\sm_{4}(\{s_2,,\ldots,s_{m}\})\right)$.
Hence, because $\{t_1,\dots,t_n\}\cap\{s_2,\dots, s_m\}=\varnothing$, we have obtained a configuration of points satisfying the hypotheses of case (i) of this proof. 
Therefore, $\vec{u}-\kappa_1\sm_4(s_1)=\vec{0},$ and by Corollary~\ref{cor:semicirclecone-one-third} of case (i) below, it follows that $\lambda_\alpha=\kappa_1$ and $\lambda_i=0$ for all $i\in\{1,\dots, n\}\setminus\alpha$. 

If $\lambda_\alpha<\kappa_1$, then $\vec{u}-\lambda_\alpha\sm_4(t_\alpha)=\sum_{i\in\{1,\dots, n\}\setminus \alpha} \lambda_i \sm_{4}(t_i)=\sum_{j=2}^m \kappa_j \sm_{4}(s_j)-\lambda_\alpha\sm_4(t_\alpha)$.
That is,
\[\vec{u}-\lambda_\alpha\sm_4(t_\alpha)=(\kappa_1-\lambda_\alpha)\sm_4(s_1) +\sum_{j=2}^m \kappa_j \sm_{4}(s_j).\]
As before, because $\left(\{t_1,\dots,t_n\}\setminus\{t_\alpha\}\right)\cap\{s_1,\dots, s_m\}=\varnothing,$ we have obtained a configuration of points satisfying the hypotheses of case (i) of this proof. 
Hence $\vec{u}-\lambda_\alpha\sm_4(t_\alpha)=\vec{0}$, and by  Corollary~\ref{cor:semicirclecone-one-third} of case (i), it follows that $\lambda_i=0$ for all $i\in\{1,\dots, n\}\setminus\alpha$.
This concludes the proof for case (ii). 

Last, observe that case (iii) follows by a similar trick: by rewriting a vector $\vec{u}$ contained in the intersection of both cones, we may obtain a configuration of points satisfying the hypotheses of case (i) or case (ii). 
\end{proof}

We emphasize that the following is a corollary of case (i) only in the proof of Proposition~\ref{prop:cones-one-third}; indeed it is used in the proof of case (ii).

\begin{corollary}\label{cor:semicirclecone-one-third} 
Let distinct $t_1,\dots, t_{n}\in S^1$ be in counterclockwise order and contained in an arc $[t_1,t_n]_{S^1}$ of length at most $\frac{2\pi}{3}$.
If $\sum_{i=1}^n \lambda_i \sm_{4}(t_i)=\vec{0}$ with $\lambda_i\ge 0$, then $\lambda_i=0$ for all $1\leq i \leq n$.
\end{corollary}

\begin{proof}
The claim is obvious in the case $n=1$.
Otherwise, because $\sm_{4}(-t)=-\sm_{4}(t),$ we may write $\sum_{i=1}^{n-1}\lambda_i \sm_{4}(t_i) = \lambda_n \sm_{4}(-t_n)$, with $-t_n\notin [t_1,t_n]_{S^1}$.
With $s_1=-t_n,$ observe that the hypotheses of case (i) of Proposition~\ref{prop:cones-one-third} are satisfied, implying \[\cone\left(\sm_{4}(\{t_1,\ldots,t_{n-1}\})\right)\cap \cone\left(\sm_{4}(-t_n)\right)=\vec{0}
\]
Since $\lambda_n \sm_{4}(-t_n)$ is in this intersection of cones, this implies $\lambda_n=0$.
Hence $\sum_{i=1}^{n-1}\lambda_i \sm_{4}(t_i)=\vec{0}$, and we may proceed iteratively to conclude $\lambda_i=0$ for all $i$.
\end{proof}

We are now ready to prove that the ``diameter non-increasing" result in Conjecture~\ref{conj:diam} is true for $r=\frac{2\pi}{3}$.

\begin{proposition}\label{prop:diam}
For $\mu\in\vrm{S^1}{\frac{2\pi}{3}}$, we have $\diam(\supp(\mu))=\diam(\supp(\mu)\cup\supp(\iota \circ p\circ \sm_4(\mu)))$.
\end{proposition}

\begin{proof}
Let $\mu=\sum_{i=1}^n \lambda_i \delta_{t_i} \in\vrm{S^1}{\frac{2\pi}{3}}$ for $t_i\in S^1$ and $\lambda_i>0$ with $\sum_i\lambda_i=1$. 
There are two cases. 
If $\{t_1,\dots,t_n\}$ are in counterclockwise order and belong to an arc of length at most $\frac{2\pi}{3}$, then  Proposition~\ref{prop:cones-one-third} implies that $\supp(\iota \circ p\circ \sm_4(\mu))\subseteq [t_1,t_n]_{S^1}$, and hence
\[\diam(\supp(\mu))=\diam(\supp(\mu)\cup\supp(\iota \circ p\circ \sm_4(\mu))).\]
Otherwise, $n=3$ and $\{t_1,t_2,t_3\}$ form the vertices of an equilateral triangle. 
In this case, we have $\iota \circ p\circ \sm_4(\mu)=\mu$ in light of Theorem~\ref{thm:faces}.
\end{proof}

\section{Conclusion}

We provide a lower bound on the diameter of a \emph{Carath\'{e}odory} set in the centrally symmetric trigonometric moment curve, i.e., a set whose convex hull contains the origin.
As applications, we obtain sharp versions of the Borsuk--Ulam theorem for maps into higher-dimensional codomains, and we gain control over the zeros of raked trigonometric polynomials.
Furthermore, we provide a geometric proof (taking advantage of continuous maps afforded by the optimal transport metric) that the Vietoris--Rips metric thickening of the circle achieves the homotopy type of the 3-sphere $S^3$ at scale parameter $r=\frac{2\pi}{3}$, in contrast to the uncountably infinite wedge-sum of 2-spheres attained by the ordinary Vietoris--Rips complex on the circle.
This proof reveals connections between Vietoris--Rips thickenings of the circle and the Barvinok--Novik orbitopes $\cB_{2k}$; analogous connections exist between \v{C}ech thickenings of the circle and the Carath\'{e}odory orbitopes.

The homotopy types of Vietoris--Rips metric thickenings of the circle $\vrm{S^1}{r}$ are currently unknown for $r>\frac{2\pi}{3}$.
To obtain Conjecture~\ref{conj:homotopy-S1}, that $\vrm{S^1}{r}\simeq\partial\cB_{2k}\cong S^{2k-1}$ for $\frac{2\pi(k-1)}{2k-1}\leq r<\frac{2\pi k}{2k+1}$, it remains to prove the homotopy equivalence $\iota\circ (p\circ \sm_{2k})\simeq \mathrm{id}_{\vrm{S^1}{r}}$, where a linear homotopy may again be well-defined (see Conjecture~\ref{conj:diam}).

\section{Acknowledgements}

We would like to thank Harrison Chapman for the insights behind the proof of Lemma~\ref{lem:det}, Micha{\l} Adamaszek, Alexander Barvinok, Yuliy Baryshnikov, Isabella Novik, Raman Sanyal, Rainer Sinn, and Cynthia Vinzant for helpful conversations, and Arkadiy Skopenkov for pointing us to~\cite{malyutin2018neighboring}.

\bibliographystyle{plain}
\bibliography{MetricThickeningsBorsukUlamTheoremsAndOrbitopes}

\appendix

\end{document}